\documentclass[12pt,reqno]{amsart}
\usepackage{mathtools,amssymb}

\usepackage{enumitem}
\setitemize{leftmargin=*}
\setenumerate{leftmargin=*,label=\textup{(\arabic*)}}
\usepackage{setspace}
\usepackage{tikz}
\usepackage{etoolbox, textcomp}
\usepackage[colorlinks=true,citecolor=red,linkcolor=blue]{hyperref}
\usepackage[normalem]{ulem}

\usepackage[capitalise,noabbrev]{cleveref}

\numberwithin{equation}{section}

\theoremstyle{plain}
\newtheorem{theorem}[subsubsection]{Theorem}
\newtheorem{lemma}[subsubsection]{Lemma}
\newtheorem{prop}[subsubsection]{Proposition}
\newtheorem{cor}[subsubsection]{Corollary}

\newtheorem*{mtheorem*}{Main Theorem}

\theoremstyle{definition}
\newtheorem{defn}[subsubsection]{Definition}

\newtheorem{remark}[subsubsection]{Remark}

\newtheorem{nota}[subsubsection]{Notation}

\Crefname{prop}{Proposition}{Propositions}
\Crefname{cor}{Corollary}{Corollaries}
\Crefname{defn}{Definition}{Definitions}

\usepackage[margin=21mm]{geometry}

\def\CC{\mathbb{C}}

\def\NN{\mathbb{Z}_{\ge 0}}

\def\QQ{\mathbb{Q}}
\def\RR{\mathbb{R}}

\def\ZZ{\mathbb{Z}}

\newcommand\cC{\mathcal{C}}

\newcommand\cF{\mathcal{F}}

\newcommand\cN{\mathcal{N}}

\def\bL{\mathbf{L}}
\def\bM{\mathbf{M}}

\newcommand\frg{\mathfrak{g}}

\newcommand\frl{\mathfrak{l}}

\newcommand\frs{\mathfrak{s}}

\newcommand\id{\textup{id}}

\newcommand{\wt}{\textup{wt}\;}

\newcommand\Hom{\textup{Hom}}

\newcommand{\Ext}{\textup{Ext}}

\renewcommand\sl{\mathfrak{sl}}

\newcommand{\btimes}{\boxtimes}

\newcommand\quash[1]{}

\newcommand{\ov}{\overline}

\newcommand\xr{\xrightarrow}

\newcommand\one{\mathbf{1}}

\renewcommand\a\alpha
\renewcommand\b\beta
\newcommand\g\gamma
\renewcommand\d\delta
\newcommand\D\Delta

\renewcommand{\th}{\theta}

\renewcommand{\l}{\lambda}

\makeatletter
\newlength{\@pxlwd} \newlength{\@rulewd} \newlength{\@pxlht}
\catcode`.=\active \catcode`B=\active \catcode`:=\active
\catcode`|=\active
\def\sprite#1(#2,#3)[#4,#5]{
   \edef\@sprbox{\expandafter\@cdr\string#1\@nil @box}
   \expandafter\newsavebox\csname\@sprbox\endcsname
   \edef#1{\expandafter\usebox\csname\@sprbox\endcsname}
   \expandafter\setbox\csname\@sprbox\endcsname =\hbox\bgroup
   \vbox\bgroup
  \catcode`.=\active\catcode`B=\active\catcode`:=\active\catcode`|=\active
      \@pxlwd=#4 \divide\@pxlwd by #3 \@rulewd=\@pxlwd
      \@pxlht=#5 \divide\@pxlht by #2
      \def .{\hskip \@pxlwd \ignorespaces}
      \def B{\@ifnextchar B{\advance\@rulewd by \@pxlwd}{\vrule
         height \@pxlht width \@rulewd depth 0 pt \@rulewd=\@pxlwd}}
      \def :{\hbox\bgroup\vrule height \@pxlht width 0pt depth
0pt\ignorespaces}
      \def |{\vrule height \@pxlht width 0pt depth 0pt\egroup
         \prevdepth= -1000 pt}
   }
\def\endsprite{\egroup\egroup}
\catcode`.=12 \catcode`B=11 \catcode`:=12 \catcode`|=12\relax
\makeatother

\def\hboxtr{\mathbin{\FormOfHboxtr}} 
\sprite{\FormOfHboxtr}(25,25)[0.5 em, 1.2 ex]

:BBBBBBBBBBBBBBBBBBBBBBBBB | :BB......................B |
:B.B.....................B | :B..B....................B |
:B...B...................B | :B....B..................B |
:B.....B.................B | :B......B................B |
:B.......B...............B | :B........B..............B |
:B.........B.............B | :B..........B............B |
:B...........B...........B | :B............B..........B |
:B.............B.........B | :B..............B........B |
:B...............B.......B | :B................B......B |
:B.................B.....B | :B..................B....B |
:B...................B...B | :B....................B..B |
:B.....................B.B | :B......................BB |
:BBBBBBBBBBBBBBBBBBBBBBBBB |

\endsprite
\sprite{\FormOfShboxtr}(25,25)[0.3 em, 0.72 ex]

:BBBBBBBBBBBBBBBBBBBBBBBBB | :BB......................B |
:B.B.....................B | :B..B....................B |
:B...B...................B | :B....B..................B |
:B.....B.................B | :B......B................B |
:B.......B...............B | :B........B..............B |
:B.........B.............B | :B..........B............B |
:B...........B...........B | :B............B..........B |
:B.............B.........B | :B..............B........B |
:B...............B.......B | :B................B......B |
:B.................B.....B | :B..................B....B |
:B...................B...B | :B....................B..B |
:B.....................B.B | :B......................BB |
:BBBBBBBBBBBBBBBBBBBBBBBBB |

\endsprite

\newcommand{\vir}{\mathcal{L}}

\newcommand{\virb}{\vir_{\ge0}}

\newcommand{\virn}{\vir_{<0}}

\newcommand{\uea}[1]{\mathcal{U}(#1)}

\newcommand{\blank}{{-}}

\newcommand{\hw}{highest-weight}

\newcommand{\oc}{\mathcal{O}_c}
\newcommand{\ocfin}{\oc^{\textup{fin}}}
\newcommand{\ocleft}{\oc^{\textup{L}}}
\newcommand{\ocright}{\oc^{\textup{R}}}
\newcommand{\cofcat}{\mathcal{C}_1}
\DeclarePairedDelimiter{\abs}{\lvert}{\rvert}

\makeatletter
\renewcommand\author@andify{%
  \nxandlist {\unskip ,\penalty-1 \space\ignorespaces}%
    {\unskip {} \@@and~}%
    {\unskip \penalty-2 \space \@@and~}%
}
\makeatother

\begin{document}

\title{Tensor Categories arising from the Virasoro Algebra}
\author[T.~Creutzig, C.~Jiang, F.~Orosz Hunziker, D.~Ridout, J.~Yang]{Thomas Creutzig, Cuipo Jiang, Florencia Orosz Hunziker,\\ David Ridout and Jinwei Yang}
\thanks{We would like to thank Yi-Zhi Huang, Robert McRae and Antun Milas for many useful discussions. We also thank the referee for their valuable comments and suggestions. T.\ C.\ is supported by NSERC $\#$RES0020460, C.\ J.\ is supported by NSFC grants 11771281 and 11531004, DR's research is supported by the Australian Research Council Discovery Project DP160101520.}

\address{(T.~Creutzig) Department of Mathematical and Statistical Sciences, University of Alberta, Edmonton, AB, Canada T6G 2G1.}
\email{creutzig@ualberta.ca}

\address{(C.~Jiang) School of Mathematical Sciences, Shanghai Jiao Tong University, Shanghai 200240, China.}
\email{cpjiang@sjtu.edu.cn}

\address{(F.~Orosz~Hunziker) Mathematics Department, University of Colorado, Boulder,  Boulder, CO 80309, USA.}
\email{florencia.orosz@colorado.edu}

\address{(D.~Ridout) School of Mathematics and Statistics, University of Melbourne, Parkville, Victoria, Australia, 3010.}
\email{david.ridout@unimelb.edu.au}

\address{(J.~Yang) Department of Mathematical and Statistical Sciences, University of Alberta, Edmonton, AB, Canada T6G 2G1.}
\email{jinwei2@ualberta.ca}


\begin{abstract}
We show that there is a braided tensor category structure on the category of $C_1$-cofinite modules for the
(universal or simple) Virasoro vertex operator algebras of
arbitrary central charge.
In the generic case of central charge $c=13-6(t+t^{-1})$, with
$t \notin \QQ$, we prove semisimplicity, rigidity and non-degeneracy and also compute
the fusion rules of this tensor category.
\end{abstract}

\maketitle

\markleft{T.~CREUTZIG, C.~JIANG, F.~OROSZ HUNZIKER, D.~RIDOUT AND J.~YANG}

\onehalfspacing

\section{Introduction}
\subsection{Conformal-field-theoretic tensor categories}
Tensor categories arising in conformal field theory have been studied since the late 1980s. Moore and Seiberg were the first to realize that an axiomatic formulation
of rational two-dimensional conformal field theory essentially leads to
what we now
call a modular tensor category \cite{MS}. A vertex operator algebra is a mathematically rigorous formulation of the chiral algebra of a two dimensional conformal field theory and, thus, it is a natural general expectation that sufficiently nice categories of modules of a given vertex operator algebra form rigid
braided tensor categories.

In particular, Moore and Seiberg predicted \cite{MS} that the category of integrable \hw{} modules for an affine Lie algebra at a fixed positive-integral level should have the structure of a rigid braided tensor category. Kazhdan and Lusztig were then the first to construct such a
structure on a certain category of modules for affine Lie algebras \cite{KL1}--\cite{KL5}. However, the level in their work was not a positive integer, but was restricted so that adding 
the dual Coxeter number did not give
a positive rational number. Then, several works, including those of Beilinson--Feigin--Mazur \cite{BFM} and Huang--Lepowsky \cite{HL7}, constructed
rigid braided tensor categories at positive-integral levels; these were 
in fact shown to be modular tensor categories.

The Huang--Lepowsky construction of these
categories
is based on the
tensor category theory for general
vertex operator algebras which they developed in \cite{HL1}--\cite{HL7} and \cite{H1}, in the rational case, and in
\cite{HLZ0}--\cite{HLZ8}, for the logarithmic case (with the latter work joint with
Zhang). This
theory
turns out to be a powerful tool to construct tensor category structures and prove results in vertex operator algebras, topology, mathematical physics and related areas. Currently, however, almost all known
examples of
braided tensor category structures on categories of modules of a vertex operator algebra are for algebras satisfying Zhu's $C_2$-cofiniteness condition.

The most obvious non-$C_2$-cofinite examples are the Heisenberg vertex operator algebras.  The category generated by their \hw{} modules, with real highest weights, was shown to be a braided tensor category (in fact a vertex tensor category) in \cite{CKLiRi}.  The first substantial examples of non-rational and non-$C_2$-cofinite vertex operator algebras are the affine vertex operator algebras at admissible levels where the category of ordinary modules has been shown to admit a braided tensor category structure \cite{CHY}.  Rigidity was also proven in the case that the Lie algebra is simply-laced \cite{C1}. The non-generic but non-admissible-level case is the most challenging and here the tensor category structure is not yet understood (for recent progress, see \cite{ACGY, CY}).

Here,
we study the
important family of vertex operator algebras coming
from representations of the Virasoro algebra. For $c \in \CC$, denote by $M(c,0)$ and $L(c, 0)$ the universal Virasoro vertex operator algebra and its unique simple quotient, respectively \cite{FZ1}.  The representation theories of both have been studied intensively in
the mathematics and physics literature. In particular, the Virasoro
minimal models correspond to $L(c_{p,q},0)$, where
\begin{equation} \label{eq:cminmod}
	c_{p, q} = 1- 6\frac{(p-q)^2}{pq}, \quad \text{for} \quad
	p, q \in \ZZ_{\ge 2} \quad \text{and} \quad \gcd\{p, q\} = 1.
\end{equation}
The vertex operator algebras $L(c_{p,q},0)$ are rational \cite{W} and $C_2$-cofinite \cite{DLM}.
Huang's general theorems \cite{H4,H5} for rational and $C_2$-cofinite vertex operator algebras then show
that the modules
of $L(c_{p,q},0)$
form
a modular tensor category.

\subsection{The problem}
In this paper, we aim to construct braided tensor structures
on certain categories of $M(c,0)$-modules for
arbitrary $c \in \CC$ (note that $M(c,0) = L(c,0)$ unless $c = c_{p, q}$ as in \eqref{eq:cminmod}). As
$M(c, 0)$ is neither rational nor $C_2$-cofinite,
constructing such tensor category structures for these types of vertex operator algebras is usually very hard due to the fact that they generally have infinitely many simple objects and may admit nontrivial extensions. Another difficult problem is to establish rigidity --- even for semisimple categories this is usually very difficult.

Given $c \in \CC$, let $H_c$ denote the set of $h \in \CC$ such that the Virasoro Verma module $V(c,h)$ of central charge $c$ and conformal weight $h$ is reducible.  Let $L(c,h)$ denote the simple quotient of $V(c,h)$.  We study the category $\ocfin$ of finite-length $M(c,0)$-modules whose composition factors are of the form $L(c,h)$, with $h \in H_c$. It turns out that the objects of this category are all $C_1$-cofinite. It is actually very natural to consider the category $\cofcat$ of lower-bounded $C_1$-cofinite modules, not only because the $C_1$-cofiniteness condition is a very minor restriction and is relatively easy to verify for familiar families of vertex operator algebras,
but also because the fusion product of two $C_1$-cofinite modules is again
$C_1$-cofinite \cite{Mi} (see also \cite{N}).
However, this category is not abelian in general because it need not be
closed under taking submodules and contragredient duals.
Furthermore, the associativity isomorphism that plays a key role in the
tensor category theory of Huang, Lepowsky and Zhang is also hard to verify for $\cofcat$.

In \cite{H5}, Huang studied the category of grading-restricted generalized modules for $C_1$-cofinite vertex operator algebras (in the sense of Li \cite{LiCofin}) and proved that they form a braided tensor category if the simple objects are $\RR$-graded, $C_1$-cofinite and there exists a positive integer $N$ such that the differences between the real parts of the highest conformal weights of the simple objects are bounded by $N$ and
the level-$N$ Zhu algebra is finite-dimensional. In particular, if the vertex operator algebra is of
positive-energy (that is, if $V_{(n)} = 0$ for $n < 0$ and $V_{(0)} = \CC\one$) and has only finitely many simple modules (as is
the case when it is $C_2$-cofinite), then the category of grading-restricted generalized modules has a braided tensor category structure. Unfortunately, Huang's conditions do
not hold for the Virasoro vertex operator algebras, mainly because of the existence of infinitely many simple modules.

\subsection{The main results}
We first show that the category $\cofcat$ of lower-bounded $C_1$-cofinite generalized $M(c,0)$-modules 
is actually the same as the category $\ocfin$ of finite-length $M(c,0)$-modules with
composition factors $L(c, h)$ for $h \in H_c$. Since it is obvious that the category $\ocfin$ is closed under taking submodules and contragredient duals, this bypasses
the difficulty of showing that $\cofcat$ is abelian directly. The identification of these two categories, together with the applicability result in \cite{H6}, also verifies the associativity isomorphism needed to invoke
the logarithmic tensor category theory of Huang, Lepowsky and Zhang. As a result, the category
$\ocfin$ has a
braided tensor category structure.

We also prove rigidity and non-degeneracy for the category $\ocfin$ when $c$ is generic, that is, when
$c = 13-6(t+t^{-1})$ with
$t \notin \QQ$.
A full subcategory $\ocleft$ of this category has been studied in \cite{FZ2} by I.~Frenkel and M.~Zhu. We first show that the category $\ocfin$ is semisimple, with infinitely many simple objects, and then compute the fusion rules of these simples.
As a consequence of these computations
and the coset realization of the Virasoro vertex operator algebra \cite{GKO,ACL}, we show that there is a braided tensor equivalence between the full subcategory $\ocleft$ and a simple current twist of the category of ordinary modules for the affine vertex operator algebra $V_{\ell}(\sl_2)$ at a certain non-rational level $\ell$ depending on $t$.  This shows that
the objects in $\ocleft$ are rigid. Since the category $\ocfin$ is generated by the simple objects in $\ocleft$ and a similar full subcategory $\ocright$ (by the fusion rules of \cref{fusionrule}), the category $\ocfin$ is rigid.

To summarize,
the main theorem of the paper is as follows.
\begin{mtheorem*}
Let $\ocfin$ denote the category of finite-length
$M(c,0)$-modules of
central charge $c=13-6(t+t^{-1})$
whose
composition factors are $C_1$-cofinite.
Then,
\begin{enumerate}
\item $\ocfin$ has a braided tensor category structure (\cref{thm:tensor}).
\item For $t \notin \QQ$, $\ocfin$ is semisimple, rigid and non-degenerate  (Theorems \ref{thm:semisimple} and \ref{thm:rigid} and Proposition \ref{prop:nondegenerate}).
\end{enumerate}
\end{mtheorem*}

\noindent Recently, the tensor structures of the categories $\ocfin$ with $t 
= 1/p, p\in\ZZ_{\geq 1}$, were 
studied by McRae and the last author in \cite{MY}, where the fusion rules and rigidity of $\ocfin$ are established. 

\subsection{Applications and future work}
Rigid braided vertex tensor categories are
 relevant in various modern problems. For example, the quantum geometric Langlands program can be related to equivalences of tensor categories of modules of $W$-algebras and affine vertex algebras, see 
\cite[Sec.~6]{AFO} for example. In fact, our rigidity proof (\cref{thm:rigid}) also proves one case of Conjecture 6.4 in \cite{AFO} for $G=SU(2)$. From a physics perspective, this relates to the $S$-duality and again vertex tensor categories are crucial as they describe categories of line defects ending on topological boundary conditions \cite{CGai, FGai}.

Vertex tensor categories also allow one to construct module
categories of vertex algebras out of those
of certain subalgebras \cite{KO, HKL, CKM}. Our results are very useful from this point of view for various important vertex algebras as they all contain the Virasoro vertex operator algebra as their subalgebra. For example, in the context of $S$-duality, vertex superalgebras that are extensions in a completion of $\ocfin$ appear in \cite{CGL}. As an instructive example, the simple affine vertex operator superalgebra $L_k(\mathfrak{osp}(1\vert2))$
at positive-integer levels \cite{CFK} and, more generally
at all admissible levels \cite{CKLRi},
has been well-understood by viewing it as an extension of the tensor product of $L_k(\mathfrak{sl}(2))$ and a
Virasoro minimal model. The results of our work allow us to extend these studies to generic levels.

More importantly, the vertex algebras with the best understood non-semisimple representation categories
are associated to extensions of the Virasoro vertex operator algebra at central charge $c_{1,p}$. These are
the triplet algebras $\mathcal W(p)$ \cite{Ka, GaK, AM1, NT, TW, FGST1}, the singlet algebras $\mathcal M(p)$ \cite{A1, AM2, CM, RiW} and the logarithmic $\mathcal B(p)$ algebras \cite{A2, CRiW, ACKR}. Note that the $\mathcal B(p)$-algebras are extensions of the Virasoro algebra times a rank one Heisenberg algebra and that $\mathcal B(2)$ is the rank one $\beta\gamma$ vertex algebra whose even subalgebra is $L_{-1/2}(\mathfrak{sl}(2))$ \cite{Ri}.

An illustration of the usefulness of tensor categories is the recent work \cite{ACGY}. There, uniqueness results of certain vertex operator algebra extensions of the Virasoro algebra at central charge $c_{1,p}$ are derived \cite[Thm.~8 and Cor.~14]{ACGY}. These results could only be proven because of our Main Theorem. Moreover, the uniqueness statements were then employed to resolve affirmatively the conjectures (see \cite{CRiW, C2, ACKR}) that the $\mathcal B(p)$ algebra is a subregular $W$-algebra of type $\mathfrak{sl}_{p-1}$ and also a chiral algebra of Argyres-Douglas theories of type $(A_1, A_{2p-3})$.

Another important application in this spirit is to construct and study rigid braided tensor category structures for these extensions of Virasoro vertex operator algebras. While braided tensor structures on the category of generalized modules for the triplet algebras $\mathcal W(p)$ have been proven to exist \cite{H5} and the rigidity and fusion rules are known 
\cite{TW}, rigid braided tensor structures on the corresponding module 
categories for the singlet vertex algebras $\mathcal M(p)$ and the $\mathcal B(p)$-algebras have only recently been constructed and studied in \cite{CMY2}.  This study combined 
the Virasoro tensor categories constructed in the present paper 
with 
direct limit completion technology for 
vertex and braided tensor categories \cite{CMY}. We expect to apply the techniques and methodology developed in the present paper, as well as in \cite{CMY,CMY2}, to study the vertex and braided tensor structures of 
the representation categories of many more vertex operator algebras.

\section{Preliminaries}
\subsection{Vertex operator algebras}
Let $(V, Y, \mathbf{1}, \omega)$ be a vertex operator algebra. We first recall the definitions of various types of $V$-modules.
\begin{defn}
\leavevmode
\begin{enumerate}
\item
A \textit{weak $V$-module} is a vector space $W$
equipped with a vertex operator map
\begin{equation}
\begin{aligned}
 Y_W \colon V &\rightarrow (\mathrm{End}\,W)[[x,x^{-1}]],\\
 v &\mapsto Y_W(v,x)=\sum_{n\in\ZZ} v_n\,x^{-n-1},
\end{aligned}
\end{equation}
satisfying the following axioms:
\begin{enumerate}[label=\textup{(\roman*)}]
\item
The lower truncation condition: for $u, v\in V$, $Y_W(u,x)v$ has only finitely many
terms with
negative powers in $x$.

 \item
 The vacuum property: $Y_W(\mathbf{1},x)$ is the identity endomorphism $1_W$ of $W$. 

 \item
 The Jacobi identity: for $u,v\in V$,
\begin{align}
 &x_0^{-1}\delta\left(\dfrac{x_1-x_2}{x_0}\right) Y_W(u,x_1)Y_W(v,x_2) \notag \\
 &- x_0^{-1}\delta\left(\dfrac{-x_2+x_1}{x_0}\right)Y_W(v,x_2)Y_W(u,x_1) \notag \\
 &\quad = x_2^{-1}\delta\left(\dfrac{x_1-x_0}{x_2}\right)Y_W(Y(u,x_0)u,x_2).
\end{align}

\item
The Virasoro algebra relations: if we write $Y_W(\omega,x)=\sum_{n\in\ZZ} L_n x^{-n-2}$,
then for any $m,n\in\ZZ$, we have
\begin{equation}
 [L_m,L_n]=(m-n)L_{m+n}+\dfrac{m^3-m}{12} \delta_{m+n,0} c \,1_W,
\end{equation}
where $c \in \CC$ is the central charge of $V$.

\item
The $L_{-1}$-derivative property: for any $v\in V$,
\begin{equation}
 Y_W(L_{-1}v,x)=\dfrac{d}{dx} Y_W(v,x).
 \end{equation}
\end{enumerate}

\item
A \textit{generalized $V$-module} is a weak $V$-module $(W, Y_{W})$ with a
$\CC$-grading
\begin{equation} \label{eq:Vmodgrading}
	W=\coprod_{n\in \CC}W_{[n]}
\end{equation}
such that $W_{[n]}$ is a
generalized eigenspace for the operator $L_{0}$ with eigenvalue $n$.

\item
A \textit{lower-bounded generalized $V$-module} is a generalized $V$-module such that for any $n\in\CC$, $W_{[n+m]}=0$ for $m\in\ZZ$ sufficiently negative.

\item A \textit{grading-restricted generalized $V$-module} is a lower-bounded generalized $V$-module such that dim $W_{[n]}<\infty$ for any $n\in\CC$.

\item
An \textit{ordinary $V$-module} (sometimes just called
a $V$-module for brevity) is a grading-restricted
generalized $V$-module such that the $W_{[n]}$ in \eqref{eq:Vmodgrading}
are eigenspaces
for the operator $L_{0}$.
\item A generalized $V$-module $W$ has {\em length} $l$ if there exists a filtration
$W = W_1 \supset \cdots \supset W_{l+1} = 0$ of generalized $V$-submodules such that each $W_i/W_{i+1}$ is
irreducible.
A {\em finite-length} generalized $V$-module is one whose length is finite.
\end{enumerate}
\end{defn}

\begin{defn}
A vertex operator algebra is \emph{rational} if every weak module is a direct sum of simple ordinary modules. We say that a category of ordinary $V$-modules is {\em semisimple} if every ordinary module is a direct sum of simple ordinary modules.
\end{defn}

Let $V$ be a vertex operator algebra and let $(W, Y_W)$ be a lower-bounded generalized $V$-module, graded as in \eqref{eq:Vmodgrading}.
Its \emph{contragredient module} is then the vector space
\begin{equation}
W' = \coprod_{n \in \CC}(W_{[n]})^*,
\end{equation}
equipped with the vertex operator map $Y'$ defined by
\begin{equation}
\langle Y'(v,x)w', w \rangle = \langle w', Y^{\circ}_W(v,x)w \rangle
\end{equation}
for any $v \in V$, $w' \in W'$ and $w \in W$, where
\begin{equation}
Y^{\circ}_W(v,x) = Y_W(e^{xL(1)}(-x^{-2})^{L_0}v, x^{-1}),
\end{equation}
for any $v \in V$, is the \emph{opposite vertex operator} (see 
\cite{FHL}).
We also use the standard notation
\begin{equation}
\ov{W} = \prod_{n \in \CC}W_{[n]},
\end{equation}
for the formal completion of $W$ with respect to the $\CC$-grading.

The notion of a logarithmic intertwining operator \cite{Mwk,HLZ2} plays a key role in the study of the representations of vertex operator algebras that are not rational, such as the Virasoro vertex operator algebras discussed in this paper. Let $W\{x\}$ denote the space of formal power series in arbitrary complex powers of $x$ with coefficients in $W$.
\begin{defn}\label{log:def}
Let $(W_1,Y_1)$, $(W_2,Y_2)$ and $(W_3,Y_3)$ be generalized $V$-modules. A \emph{logarithmic intertwining
operator of type $\binom{W_3}{W_1\,W_2}$} is a linear map
\begin{equation}\label{log:map0}
\begin{aligned}
\mathcal{Y}\colon W_1\otimes W_2&\to W_3\{x\}[\log x], \\
w_{(1)}\otimes w_{(2)}&\mapsto{\mathcal{Y}}(w_{(1)},x)w_{(2)}
=\sum_{k=0}^{K}\sum_{n\in{\CC}}\left({w_{(1)}}_{n, k}^{\mathcal{Y}}w_{(2)}\right)x^{-n-1}(\log x)^{k},
\end{aligned}
\end{equation}
satisfying
the following conditions:
\begin{enumerate}[label=\textup{(\roman*)}]
\item
The lower truncation condition: given
any $w_{(1)}\in W_1$, $w_{(2)}\in W_2$, $k=0, \dots, K$ and $n\in{\CC}$,
\begin{equation}
{w_{(1)}}_{n+m, k}^{\mathcal{Y}}w_{(2)}=0\quad\text{for}\ m\in\ZZ\ \text{sufficiently large}.
\end{equation}
\item
The Jacobi identity:
\begin{align}\label{Jacob}
&x^{-1}_0\delta \bigg( {x_1-x_2\over x_0}\bigg)
Y_3(v,x_1){\mathcal{Y}}(w_{(1)},x_2)w_{(2)}\nonumber \\
&- x^{-1}_0\delta \bigg( {x_2-x_1\over -x_0}\bigg)
{\mathcal{Y}}(w_{(1)},x_2)Y_2(v,x_1)w_{(2)}\nonumber \\
&\quad = x^{-1}_2\delta \bigg( {x_1-x_0\over x_2}\bigg){
\mathcal{Y}}(Y_1(v,x_0)w_{(1)},x_2) w_{(2)},
\end{align}
for $v\in V$, $w_{(1)}\in W_1$ and $w_{(2)}\in W_2$.
\item
The $L_{-1}$-derivative property: for any $w_{(1)}\in W_1$,
\begin{equation}
{\mathcal{Y}}(L_{-1}w_{(1)},x)=\frac d{dx}{\mathcal{Y}}(w_{(1)},x).
\end{equation}
\end{enumerate}
A logarithmic intertwining operator $\mathcal{Y}$ is called an \textit{intertwining operator} if no $\log x$ appears in ${\mathcal{Y}}(w_{(1)},x)w_{(2)}$, for any $w_{(1)}\in W_1$ and $w_{(2)}\in W_2$. The dimension of the space  $\mathcal{V}_{W_1 W_2}^{W_3}$ of all logarithmic intertwining operators is called the \emph{fusion coefficient} or \textit{fusion rule} of the same type; it is denoted by $\mathcal{N}_{W_1 W_2}^{W_3}$.
\end{defn}

We remark that the term ``fusion rule'' commonly has a different meaning in the literature, referring instead to an explicit identification of the isomorphism class of a ``fusion product''.

The following finiteness condition plays an important role in this paper:
\begin{defn}
Let $V$ be a vertex operator algebra and let $W$ be a weak $V$-module. Let $C_1(W)$ be the subspace of $W$ spanned by elements of the form $u_{-1}w$, where $u\in V_+:=\coprod_{n\in \ZZ_{>0}} V_{(n)}$ and $w\in W$. We say that $W$ is \textit{$C_1$-cofinite} if the \textit{$C_1$-quotient} $W/C_1(W)$ is finite-dimensional.
\end{defn}
\noindent We shall also need the following facts.
\begin{lemma}[\cite{H1}] \label{c1cofinite}
Let $V$ be a vertex operator algebra.
\begin{enumerate}[label=\textup{(\arabic*)}]
	\item For a (generalized) $V$-module $W$ and a (generalized) $V$-submodule $U$ of $W$, \label{qu}
	\begin{equation}
		C_1(W/U)=\frac{C_1(W)+U}{U} \quad \text{and so} \quad \frac{W/U}{C_1(W/U)}\cong\frac{W}{C_1(W)+U}.
	\end{equation}
	It follows that $W/U$ is $C_1$-cofinite if $W$ is.
	\item If $W$ is a finite-length $V$-module with $C_1$-cofinite 
	composition factors, then $W$ is also $C_1$-cofinite. \label{inclusion}
\end{enumerate}
\end{lemma}

\subsection{Representations of Virasoro algebra} \label{sec:vir}
The Virasoro algebra is the Lie algebra
\begin{equation}
\vir = \bigoplus_{n=-\infty}^{+\infty}\CC L_n \oplus \CC\underline{c}
\end{equation}
 with the commutation relations
\begin{equation}
[L_m, L_n] = (m-n)L_{m+n}+\frac{m^3-m}{12}\delta_{m+n,0}\underline{c}, \quad [\vir, \underline{c}] = 0.
\end{equation}
Set
\begin{equation}
\virb=\bigoplus_{n\ge 0} \CC L_n\oplus \CC\underline{c}
\end{equation}
and recall that for $c,h \in \CC$, the Verma module $V(c,h)$ for
$\vir$ is defined to be
\begin{align}
V(c,h) = \uea{\vir}\otimes_{\uea{\virb}}\CC\one_{c, h},
\end{align}
where $\uea{\blank}$ denotes a universal enveloping algebra and the $\virb$-module structure of
$\CC\one_{c,h}$ is given by
$L_0 \one_{c,h}=h\one_{c,h}$, $\underline{c}\one_{c,h}=c \one_{c,h}$ and
$L_n \one_{c,h}=0$ for $n > 0$. As usual, the Verma module $V(c,h)$ has a
unique (possibly trivial) maximal proper submodule.  We denote its unique simple quotient
by $L(c,h)$.

When $h=0$, $L_{-1}\one_{c,0}$ is a singular vector in $V(c,0)$ (the tensor product symbol is omitted for brevity). It was shown in \cite{FZ1} that
\begin{equation}
M(c,0) = V(c,0)/\langle L_{-1}\one_{c,0}\rangle
\end{equation}
admits the structure of a vertex operator algebra.  It is called the \emph{universal Virasoro vertex operator algebra} of central charge $c$.  The simple quotient
$L(c,0)$ of $M(c,0)$ therefore admits a
vertex operator algebra structure as well.  Note that all $L(c,0)$- and $M(c,0)$-modules are $\vir$-modules.  Another important result of \cite{FZ1} is that every \hw{} $\vir$-module of central charge $c$ is an $M(c,0)$-module.

We recall the
existence criterion of Feigin and Fuchs for
singular vectors in Verma $\vir$-modules.
Useful expositions may be found in \cite{IK, KR}.
Note that a Verma module is reducible if and only if it possesses a non-trivial singular vector (by which we mean one that is not proportional to the cyclic \hw{} vector).
\begin{prop}[\cite{FF}]\label{thm:vermared}
For
$r,s\in \ZZ_{\ge1}$ and $t \in \CC\setminus\{0\}$, define
\begin{equation} \label{eq:virpara}
c=c(t)=13-6t-6t^{-1}, \quad
h=h_{r,s}(t)=\frac{r^2-1}{4}t-\frac{rs-1}{2}+\frac{s^2-1}{4}t^{-1}.
\end{equation}
\begin{enumerate}
	\item If there exist $r,s\in \ZZ_{\ge1}$ and $t \in \CC\setminus\{0\}$ such that $c$ and $h$ satisfy \eqref{eq:virpara}, then there is a singular vector of weight $h+rs$ in the Verma module $V(c,h)$.
	\item Conversely, if $V(c,h)$ possesses a non-trivial singular vector, then there exist $r, s \in \ZZ_{\ge1}$ and $t \in \CC \backslash \{0\}$ such that \eqref{eq:virpara} holds.
	\item For each $N$, there exists at most one singular vector of weight $h+N$ in $V(c,h)$, up to scalar multiples.
\end{enumerate}
\end{prop}

Of course, any singular vector in $V(c,h)$ may be expressed as a linear combination of Poincar\'{e}--Birkhoff--Witt-ordered monomials in the $L_i$, $i<0$, acting on the \hw{} vector $\one_{c,h}$.  A crucial fact for what follows is that the coefficient of $L_{-1}^N$, is never $0$ (irrespective of the chosen order).  Here, $N$ is the conformal weight of the singular vector minus that of $\one_{c,h}$.

\begin{prop}[\cite{As}]\label{sing}
If $V(c,h)$ has a
singular vector $v$ of conformal weight $h+N$,
then
\begin{align}
v=\sum_{\abs{I}=N}a_{I}(c,h)L_{-I}\one_{c,h},
\end{align}
where $L_{-I}=L_{-i_{1}}\cdots L_{-i_{n}}$ and the sum is over sequences $I=\{i_1, \dots, i_n\}$ of ordered $n$-tuples $i_1\ge \dots\ge i_n$ with $\abs{I}=i_1+\dots+i_n=N$. Moreover, the
coefficients $a_{I}(c,h)$ depend polynomially on $c$ and $h$ and the coefficient $a_{\{1,\dots,1\}}(c,h)$ of $L_{-1}^{N}$ may be chosen to be $1$.
\end{prop}

\begin{nota}\label{c}
For $c = c(t) \in \CC$, set
\begin{equation}
H_c = \left\{h \mid V(c,h)\ \text{is reducible}\right\}.
\end{equation}
By \cref{thm:vermared}, we have $H_c = \{h_{r,s}(t) \mid r,s \in \ZZ_{\ge1}\}$.
For brevity, we shall also denote the Verma module $V(c(t),h_{r,s}(t))$ and its simple quotient by $V_{r,s}$
and $L_{r,s}$, respectively.
\end{nota}
\noindent The set $H_c$ of conformal weights is often referred to as the \emph{extended Kac table} because the original Kac table of the Virasoro minimal models corresponds to the subset of $h_{r,s}(t)$ with $t=\frac{p}{q}$ ($p,q \in \ZZ_{\ge2}$ and $\gcd\{p,q\}=1$), $r=1,\dots,p-1$ and $s=1,\dots,q-1$.  This subset consists precisely of the conformal weights of the simple $L(c(t),0)$-modules \cite{W,RiW2}.

The embedding structure of Virasoro Verma modules
is also due to Feigin and Fuchs.  A convenient summary appears in \cite[Ch.~5]{IK}.

\begin{theorem}[\cite{FF}]\label{exhaustiveembedding}
The embedding structures of the reducible Verma $\vir$-modules $V_{r,s}$, where $r,s \in \ZZ_{\ge1}$, are as follows: 
\begin{enumerate}
\item \label{minmodembed} If $t = \frac{q}{p}$, for $p, q \in \ZZ_{\ge 1}$ and $\gcd\{p, q\} = 1$, then there are two possible ``shapes'' for the embedding diagrams (see \cite{IK} for further details):
\noindent If $r$ is a multiple of $p$ or $s$ is a multiple of $q$, then one has the embedding chain
\begin{equation}
V_{r,s} \longleftarrow \bullet \longleftarrow \bullet \longleftarrow \bullet \longleftarrow \bullet \longleftarrow \cdots.
\end{equation}
Otherwise, the embedding diagram is as follows:
\begin{equation}
\begin{tikzpicture}[->,>=latex,scale=1.4]
\node (b0) at (-0.5,1/2){$V_{r,s}$};
\node (b1) at (1, 0){$\bullet$};
\node (a1) at (1, 1){$\bullet$};
\node (b2) at (2.5, 0){$\bullet$};
\node (a2) at (2.5, 1){$\bullet$};
\node (b3) at (4, 0){$\bullet$};
\node (a3) at (4, 1){$\bullet$};
\node (b4) at (5.5, 0){$\bullet$};
\node (a4) at (5.5, 1){$\bullet$};
\node (b5) at (7, 0){$\cdots$};
\node (a5) at (7, 1){$\cdots$};
\draw[] (b1) -- node[left]{} (b0);
\draw[] (a1) -- node[left]{} (b0);
\draw[] (b2) -- node[left]{} (b1);
\draw[] (b2) -- node[left]{} (a1);
\draw[] (a2) -- node[left]{} (b1);
\draw[] (a2) -- node[left]{} (a1);
\draw[] (b3) -- node[left]{} (b2);
\draw[] (b3) -- node[left]{} (a2);
\draw[] (a3) -- node[left]{} (b2);
\draw[] (a3) -- node[left]{} (a2);
\draw[] (b4) -- node[left]{} (b3);
\draw[] (b4) -- node[left]{} (a3);
\draw[] (a4) -- node[left]{} (b3);
\draw[] (a4) -- node[left]{} (a3);
\draw[] (b5) -- node[left]{} (b4);
\draw[] (b5) -- node[left]{} (a4);
\draw[] (a5) -- node[left]{} (b4);
\draw[] (a5) -- node[left]{} (a4);
\end{tikzpicture}.
\end{equation}
\item If $t = -\frac{q}{p}$ for $p, q \in \ZZ_{\ge 1}$ and $\gcd\{p, q\} = 1$, then there are again two possible ``shapes'',
similar to those in \ref{minmodembed}
except that the diagrams are now finite (again the details may be found in \cite{IK}):

\noindent If $r$ is a multiple of $p$ or $s$ is a multiple of $q$, then one has the embedding chain
\begin{equation}
{V_{r,s}} \longleftarrow \bullet \longleftarrow \bullet \longleftarrow \bullet \longleftarrow \cdots \longleftarrow \bullet\ .
\end{equation}
Otherwise, the embedding diagram is as follows.
\begin{equation}
\begin{tikzpicture}[->,>=latex,scale=1.4]
\node (b0) at (-0.5,1/2){$V_{r,s}$};
\node (b1) at (1, 0){$\bullet$};
\node (a1) at (1, 1){$\bullet$};
\node (b2) at (2.5, 0){$\bullet$};
\node (a2) at (2.5, 1){$\bullet$};
\node (bm) at (4,0){$\cdots$};
\node (am) at (4,1) {$\cdots$};
\node (b3) at (5.5, 0){$\bullet$};
\node (a3) at (5.5, 1){$\bullet$};
\node (b4) at (7, 0){$\bullet$};
\node (a4) at (7, 1){$\bullet$};
\node (a5) at (8.5, 0.5){$\bullet$};
\draw[] (b1) -- node[left]{} (b0);
\draw[] (a1) -- node[left]{} (b0);
\draw[] (b2) -- node[left]{} (b1);
\draw[] (b2) -- node[left]{} (a1);
\draw[] (a2) -- node[left]{} (b1);
\draw[] (a2) -- node[left]{} (a1);
\draw[] (bm) -- node[left]{} (b2);
\draw[] (bm) -- node[left]{} (a2);
\draw[] (am) -- node[left]{} (b2);
\draw[] (am) -- node[left]{} (a2);
\draw[] (b3) -- node[left]{} (bm);
\draw[] (b3) -- node[left]{} (am);
\draw[] (a3) -- node[left]{} (bm);
\draw[] (a3) -- node[left]{} (am);
\draw[] (b4) -- node[left]{} (b3);
\draw[] (a4) -- node[left]{} (b3);
\draw[] (b4) -- node[left]{} (a3);
\draw[] (a4) -- node[left]{} (a3);
\draw[] (a5) -- node[left]{} (b4);
\draw[] (a5) -- node[left]{} (a4);
\end{tikzpicture}.
\end{equation}
\item If $t \notin \QQ$, then the embedding diagram is $V_{r,s} \longleftarrow V_{-r, s}$. 
\end{enumerate}

\end{theorem}

\begin{cor} \label{cor:embeddings}
	It follows from these embedding structures that:
	\begin{enumerate}
		\item Every non-zero submodule of a Verma $\vir$-module is either a Verma module itself or the sum of two Verma modules. \label{it:vermasubmodules}
		\item Every non-Verma \hw{} $\vir$-module has finite length. \label{it:hwfinlen}
		\item Every Verma module $V(c,h)$ in the embedding diagram of a reducible Verma $\vir$-module satisfies $h \in H_c$, with the exception of the socle which, if it exists, has the form $V(c,h')$ for some $h' \notin H_c$.
		\item $M(c,0) = L(c,0)$ unless the central charge satisfies $c = c(\frac{p}{q}) = c_{p, q}$, where $p, q \in \ZZ_{\ge 2}$ and $\gcd\{p,q\} = 1$ (see \eqref{eq:cminmod}). \label{it:M=L}
	\end{enumerate}
\end{cor}
\noindent In light of \cref{cor:embeddings}\ref{it:M=L} and the fact that the $L(c_{p,q},0)$-modules ($p, q \in \ZZ_{\ge 2}$ and $\gcd\{p,q\} = 1$) form a modular tensor category \cite{H4,H5}, we shall restrict our considerations in what follows to modules of the universal Virasoro vertex operator algebras $M(c,0)$ (for arbitrary $c \in \CC$).

\begin{remark}
The representations of $M(c,0)$
have been investigated in both the mathematics and physics literature. We list some relevant work here.
\begin{enumerate}
\item
The representation theory of $M(c,0)$ has been explored in detail by physicists under the moniker ``logarithmic minimal models'' \cite{PRZ,RS,MRR}.  In particular, fusion products were studied \cite{GKind,EF,MR} in order to determine the types of non-semisimple modules that appeared at central charges of the form $c_{p,q}$, $p,q \in \ZZ_{\ge1}$, especially the so-called \emph{staggered modules} \cite{R,KRi} that arise in logarithmic conformal field theories \cite{RoS,G,F2,Ga,CR}.
\item
The famous triplet algebras $\mathcal W(p)$ \cite{Ka, F1, GaK, AM1, NT, TW, FGST1} and singlet algebras $\mathcal M(p)$ \cite{A1, AM2, CM, RiW} have been extensively studied as extensions of $M(c_{1,p},0)$.  The $\mathcal B(p)$-algebras \cite{A2, CRiW, ACKR} are similarly extensions of $M(c_{1,p},0)$ times a rank one Heisenberg vertex operator algebra.
\item
The representations and fusion rules of the Virasoro algebra of central charge $c_{1,1} = 1$
have been studied in \cite{M}. In \cite{Mc}, these fusion rules were used to prove that the semisimple full subcategory of $M(1,0)$-modules generated by the $L(1, \frac{n^2}{4})$, $n \in \NN$, is braided-tensor and tensor equivalent to a modification of the category of finite-dimensional modules of $\mathfrak{sl}_2$ involving $3$-cocycle on $\ZZ/2\ZZ$.
\item
For the case $t \notin \QQ$, the category of modules generated by the $V_{1,s}$ with $s \in \ZZ_{\ge1}$ was studied in \cite{FZ2} and the fusion rules of this category were determined.  These rules will be crucial to our results, see \cref{sec:fusion} below.
\end{enumerate}
\end{remark}

The vertex operator algebra $M(c, 0)$ is neither rational nor $C_2$-cofinite because it has infinitely many inequivalent simple
modules. It is therefore natural to study the $C_1$-cofiniteness
of an $M(c, 0)$-module $W$.  One can check that
\begin{equation} \label{eq:C1Vir}
	C_1(W) = \sum_{n=2}^{\infty} L_{-n} W,
\end{equation}
hence the $C_1$-quotient of a \hw{} $M(c,0)$-module is spanned by the powers of $L_{-1}$ acting on the cyclic \hw{} vector.  We therefore have the following
consequences of \cref{sing,cor:embeddings}.
\begin{cor}\label{noc1}
\leavevmode
\begin{enumerate}
	\item A \hw{} $M(c,0)$-module is $C_1$-cofinite if and only if it is not isomorphic to a Verma module.  In particular, $L(c, h)$ is $C_1$-cofinite if and only if $h \in H_c$. \label{it:c1notverma}
	\item A $C_1$-cofinite \hw{} $M(c,0)$-module
	has finite length. \label{hig}
	\item The composition factors of a $C_1$-cofinite \hw{} $M(c,0)$-module are of the form $L(c,h)$ with $h \in H_c$. \label{it:c1compfacts}
\end{enumerate}
\end{cor}

\section{Equivalence of two categories} \label{sec:equiv}
Recall that $\cofcat$ denotes
the category of lower-bounded $C_1$-cofinite generalized $M(c,0)$-modules
and that
$\ocfin$ denotes
the category of finite-length $M(c,0)$-modules with
composition factors $L(c, h)$ for $h \in H_c$. In this section, we will prove that these two categories are the same.

From \cref{c1cofinite}\ref{inclusion} and \cref{noc1}, we have $\ocfin \subseteq \cC_1$. We therefore only need to prove the reverse inclusion in what follows.

\subsection{The reverse inclusion}
\begin{lemma}\label{easylemma}
Let $V$ be a vertex operator algebra and let
$W$ be a generalized $V$-module.
\begin{enumerate}
	\item Suppose that $0 \neq w \in W$ has $L_0$-eigenvalue $h \in \CC$ and that $h-n$ is not an eigenvalue of $L_0$ for any $n \in \ZZ_{\ge1}$.  Then, $w \notin C_1(W)$. \label{easy1}
	\item If $W$ is a lower-bounded generalized $V$-module satisfying $W=C_1(W)$, then $W=0$. \label{easy2}
\end{enumerate}
\end{lemma}
\begin{proof}
If $w \in C_1(W)$, then there would exist homogeneous elements $u^i \in M(c,0)$, with $\wt u_i \in \ZZ_{\ge1}$, and $w^i\in W$ such that $w=\sum_{i} u^i_{-1}w^i$. As $w\neq 0$, there exists at least one non-zero term $u^k_{-1}w^k$ in this decomposition. But, $\wt w^k = \wt w - \wt u^k \in h - \ZZ_{\ge1}$, a contradiction.  This proves \ref{easy1}.

For \ref{easy2}, assume that $W \neq 0$.  As $W$ is lower-bounded, there exists an $L_0$-eigenvalue $h \in \CC$ such that $h-n$ is not an eigenvalue of $L_0$ for any $n \in \ZZ_{\ge1}$.  If $w$ is the corresponding eigenvector, then $w \in W = C_1(W)$, contradicting \ref{easy1}.
\end{proof}

\begin{prop}\label{mainprop1}
If $W \in \cC_1$, then
there exists a finite filtration
\begin{equation}\label{filh}
0 = W_0 \subset W_1 \subset \dots \subset W_{n-1} \subset W_{n} = W
\end{equation}
such that
the quotients $W_i/W_{i-1}$,
$i = 1, \dots, n$, are \hw{} $M(c,0)$-modules.
\end{prop}
\begin{proof}
Since $W$ is lower-bounded, we can choose an $L_0$-eigenvector $w_1 \in W$ whose eigenvalue $h$ is such that there are no $L_0$-eigenvalues of the form $h-n$, $n\in\ZZ_{\ge1}$.  It follows that $w_1$ is a \hw{} vector; moreover, $w_1 \notin C_1(W)$ by \cref{easylemma}\ref{easy1}.
Denote by $W_1$ the \hw{} submodule of $W$ generated by $w_1$.
From \cref{c1cofinite}\ref{qu},
$W/W_1$ is also a lower-bounded $C_1$-cofinite generalized module.  Moreover, $w_1 \notin C_1(W)$ gives $C_1(W) \subset C_1(W)+W_1$ and hence
\begin{equation}
	\dim \frac{W/W_1}{C_1(W/W_1)} = \dim \frac{W}{C_1(W)+W_1} < \dim \frac{W}{C_1(W)} < \infty.
\end{equation}

The same argument now shows that
$W/W_1$ has a \hw{} vector
\begin{equation} \label{notinthesum}
	\overline{w_2} \notin C_1(W/W_1) = \frac{C_1(W)+W_1}{W_1} \quad \Rightarrow \quad w_2 \notin C_1(W)+W_1,
\end{equation}
where $w_2 \in W$ is any element whose image in $W/W_1$ is $\overline{w_2}$.  Now, $\overline{w_2}$ generates a
\hw{} module $\overline{W_2} \subseteq
W/W_1$.
It follows that there exists $W_2\subseteq W$ such that $W_2/W_1\cong \overline{W_2}$ is \hw{} and $W_1\subset W_2$ gives $C_1(W)+W_1 \subset C_1(W)+W_2$, by \eqref{notinthesum}.  In other words,
we have obtained a sequence of epimorphisms
\begin{align}
W\twoheadrightarrow W/W_1\twoheadrightarrow (W/W_1)/(W_2/W_1)\cong W/ W_2
\end{align}
satisfying $\dim W / (C_1(W)+W_2) < \dim W / (C_1(W)+W_1)$, hence
\begin{equation}
	\dim \frac{W/W_2}{C_1(W/W_2)} < \dim \frac{W/W_1}{C_1(W/W_1)} < \dim \frac{W}{C_1(W)} < \infty.
\end{equation}

Continuing in this manner, we obtain
sequences
\begin{align}
W\twoheadrightarrow W/W_1\twoheadrightarrow W/ W_2 \twoheadrightarrow W/W_3 \twoheadrightarrow \cdots \twoheadrightarrow W/W_{n-1}\twoheadrightarrow W/W_{n}
\end{align}
in which each $W/W_i$ is lower-bounded, $C_1$-cofinite and generalized.  Moreover, the dimensions of the $C_1$-quotients of the $W/W_i$ are strictly decreasing as $i$ increases.  As
$\dim W/C_1(W)$ is finite, there exists $n$ such that this dimension is $0$.
By \cref{easylemma}, we therefore have
$W/W_{n}=0$.
Thus, we have obtained a sequence of submodules
\begin{align}
0\hookrightarrow W_1 \hookrightarrow W_2 \hookrightarrow \cdots \hookrightarrow W_{n-1}\hookrightarrow W_{n} = W
\end{align}
such that $W_{i}/W_{i-1}\cong \overline{W_{i}}$ is a \hw{} $M(c,0)$-module.
\end{proof}

\cref{mainprop1} shows that an arbitrary $W \in \cC_1$ is composed of finitely many \hw{} modules, but this does not guarantee that $W$ has finite length
because the length of one of the \hw{} modules might be infinite.  We therefore need the following stronger result whose proof is deferred until the following \lcnamecref{mainproof}.

\begin{prop}\label{mainprop2}
Let $W \in \cC_1$ have a finite filtration \eqref{filh}
such that
the quotients $W_i/W_{i-1}$, for $i=1,\dots,n$, are \hw{} $M(c,0)$-modules.
Then, each $W_i$, and hence each $W_i/W_{i-1}$, is $C_1$-cofinite.
\end{prop}
\noindent Given this \lcnamecref{mainprop2}, the main \lcnamecref{virc} of this \lcnamecref{sec:equiv}, which plays the key role in the construction of the tensor category structure on $\ocfin$, is easily proven.

\begin{theorem}\label{virc}
As categories of $M(c,0)$-modules, $\cC_1 = \ocfin$.
\end{theorem}
\begin{proof}
As remarked above, we only need to prove that $\cC_1 \subseteq \ocfin$.
Given $W \in \cC_1$,
the \hw{} modules of \cref{mainprop1}
are $C_1$-cofinite, by \cref{mainprop2},
so they are finite-length, by \cref{noc1}\ref{hig}.  It follows that $W$ is also finite-length.  Moreover, the composition factors of the \hw{} modules will all have the form $L(c,h)$, with $h \in H_c$, by \cref{noc1}\ref{it:c1compfacts}, hence so will those of $W$.  We conclude that $W \in \ocfin$, completing the proof.
\end{proof}
\noindent It only remains to prove \cref{mainprop2}.

\subsection{Proof of \cref{mainprop2}} \label{mainproof}

The hard work needed for this proof is isolated below as \cref{mainprop1.5}.  For this, we consider $\vir$-modules $\widetilde{W} \subseteq W$ for which $W/\widetilde{W}$ is \hw{}.  We prepare some convenient notation for what follows.

Let $\overline{w} \in W/\widetilde{W}$ be the cyclic \hw{} vector and let $w \in W$ be any element whose image in $W/\widetilde{W}$ is $\overline{w}$.  Since $W/\widetilde{W} = \uea{\virn} \overline{w}$, where $\virn = \bigoplus_{n<0} \CC L_n$ is the negative Virasoro subalgebra, we have
\begin{subequations} \label{eq:happy}
\begin{equation} \label{eq:happy0}
	W = \uea{\virn} w + \widetilde{W},
\end{equation}
as vector spaces, hence
\begin{equation} \label{eq:happy1}
	C_1(W) = \sum_{n=2}^{\infty} L_{-n} \uea{\virn} w + C_1(\widetilde{W}),
\end{equation}
\end{subequations}
by \eqref{eq:C1Vir}.

\begin{lemma} \label{mainprop1.5}
	Given $W \in \cC_1$, $\widetilde{W} \subseteq W$ and $w \in W$ as above, every element of $\uea{\virn} w \cap \widetilde{W}$ of sufficiently large conformal weight is in $C_1(\widetilde{W})$.
\end{lemma}
\begin{proof}
	Because $W \in \cC_1$, we have $W/\widetilde{W} \in \cC_1$, by \cref{c1cofinite}\ref{qu}, and so $W/\widetilde{W}$ cannot be isomorphic to a Verma module, by \cref{noc1}\ref{it:c1notverma}.  \cref{exhaustiveembedding} therefore allows only two possibilities: either $W/\widetilde{W}$ is isomorphic to $V(c,h)/V(c,h')$ or $V(c,h) / \bigl( V(c,h_1) + V(c,h_2) \bigr)$, for some $h, h', h_1, h_2 \in \CC$.

	Assume first that $W/\widetilde{W} \cong V(c,h) / V(c,h')$.  Then, there exists $U \in \uea{\virn}$ of weight $h'-h$ such that $U \one_{c,h}$ is singular in $V(c,h)$.  Thus, there exists $w \in W$ such that $U \overline{w} = 0$ and so $Uw \in \widetilde{W}$.  Moreover, we have
	\begin{equation} \label{eq:happy2}
		\uea{\virn} w \cap \widetilde{W} = \uea{\virn} Uw.
	\end{equation}
	For convenience, we shall normalise $U$ so that the coefficient of $L_{-1}^{h'-h}$ is $1$, as in \cref{sing}.

	Since $W \in \cC_1$, we have $L_{-1}^M w \in C_1(W)$ for all $M$ sufficiently large.  By \eqref{eq:happy1}, there exist $U^{(n)} \in \uea{\virn}$ and $w' \in C_1(\widetilde{W})$ such that
	\begin{gather}
		L_{-1}^M w = \sum_{n=2}^{\infty} L_{-n} U^{(n)} w + w' \label{eq:defw'} \\
		\Rightarrow \quad w' = L_{-1}^M w - \sum_{n=2}^{\infty} L_{-n} U^{(n)} w \in \uea{\virn} w \cap \widetilde{W} = \uea{\virn} Uw,
	\end{gather}
	by \eqref{eq:happy2}.  It follows that $w' = U'Uw \in C_1(\widetilde{W})$, for some $U' \in \uea{\virn}$ of weight $N=M-h'+h$.

	As $U'Uw = w' = L_{-1}^M w + \cdots$, the coefficient of $L_{-1}^N$ in $U'$ is $1$ and so we may write $U'$ in the form $L_{-1}^N + U''$, where $U'' \in \sum_{n=2}^{\infty} L_{-n} \uea{\virn}$.  Since $U''Uw$ is obviously in $C_1(\widetilde{W})$, it then follows that $L_{-1}^N Uw = w' - U''Uw \in C_1(\widetilde{W})$ for all sufficiently large $N$.  Consequently, every element of the form $U_N Uw \in \uea{\virn} Uw = \uea{\virn} w \cap \widetilde{W}$ is guaranteed to be in $C_1(\widetilde{W})$ if the weight of $U_N \in \uea{\virn}$ is sufficiently large, as desired.

	It remains to describe the modifications needed when $W/\widetilde{W} \cong V(c,h) / \bigl( V(c,h_1) + V(c,h_2) \bigr)$.  First, the role of $U$ is now played by two elements $U_i = L_{-1}^{h_i-h} + \cdots \in \uea{\virn}$, $i=1,2$, so that
	\begin{equation} \label{eq:happy2'}
		\uea{\virn} w \cap \widetilde{W} = \uea{\virn} U_1 w + \uea{\virn} U_2 w.
	\end{equation}
	The element $w' \in C_1(\widetilde{W})$, defined by \eqref{eq:defw'}, therefore has the form $U_1' U_1 w + U_2' U_2 w$, where $U_i' \in \uea{\virn}$ has weight $N_i = M-h_i+h$, $i=1,2$.

	The next step is slightly different 
	because comparing with \eqref{eq:defw'} leads to $U_i' = a_i L_{-1}^{N_i} + U_i''$, $i=1,2$, where $a_1, a_2 \in \CC$ satisfy $a_1+a_2=1$.  Thus, we can only conclude that
	\begin{equation} \label{eq:sad}
		a_1 L_{-1}^{N_1} U_1 w + a_2 L_{-1}^{N_2} U_2 w \in C_1(\widetilde{W}),
	\end{equation}
	for all sufficiently large $M$.  However, the embedding diagrams of \cref{exhaustiveembedding} show that when a submodule of $V(c,h)$ is not Verma, then the two generating singular vectors (here of weights $h_1$ and $h_2$) have a common descendant singular vector (of weight $h_3$ say).  Since Verma modules are free as $\uea{\virn}$-modules, this means that there exist $T_1, T_2 \in \uea{\virn}$ such that $T_1 U_1 = T_2 U_2$.  Moreover, \cref{sing} gives $T_i = L_{-1}^{h_3-h_i} + \cdots$ as usual.

	Assuming that $M$ is taken sufficiently large, it now follows that $a_2 L_{-1}^{h_3-h_2} U_2 w$ may be replaced in \eqref{eq:sad} by $a_2 T_2 U_2 w = a_2 T_1 U_1 w$, modulo terms in $C_1(\widetilde{W})$.  In other words, we arrive at $L_{-1}^N U_1 w \in C_1(\widetilde{W})$ for all $N$ sufficiently large and, by swapping the indices $1$ and $2$ in this argument, also $L_{-1}^N U_2 w \in C_1(\widetilde{W})$ for all $N$ sufficiently large.  By virtue of \eqref{eq:happy2'}, the proof is complete.
\end{proof}

We can now prove \cref{mainprop2}.  In the filtration \eqref{filh}, $W_n = W \in \cC_1$, so it will suffice to show that $W_i \in \cC_1$ implies that $W_{i-1} \in \cC_1$.  We therefore assume that $W_i \in \cC_1$.  As above, let $\overline{w_j} \in W_j / W_{j-1}$ be the cyclic \hw{} vector, for each $1 \le j \le n$, and choose $w_j \in W_j$ so that its image in $W_j/W_{j-1}$ is $\overline{w_j}$.

As $w_j \in W_i$ for each $j<i$, we have $L_{-1}^{N_j} w_j \in C_1(W_i)$ for all sufficiently large $N_j$.  By \eqref{eq:happy1}, we may therefore write
\begin{equation}
	L_{-1}^{N_j} w_j = \sum_{n\ge2} L_{-n} U^{(n)}_j w_i + w'_j,
\end{equation}
where $U^{(n)}_j \in \uea{\virn}$ and $w'_j \in C_1(W_{i-1})$.  The first term on the right-hand side is clearly in $\uea{\virn} w_i$.  However, it is also in $W_{i-1}$ because the second term is, as is the left-hand side (because $j<i$).  By \cref{mainprop1.5} (with $W=W_i$, $\widetilde{W}=W_{i-1}$ and $w=w_i$), this first term therefore belongs to $C_1(W_{i-1})$ for sufficiently large $N_j$.  But, the second term does too, hence we have
\begin{equation} \label{eq:happy3}
	L_{-1}^{N_j} w_j \in C_1(W_{i-1})
\end{equation}
for all $j<i$ and sufficiently large $N_j$.

Iterating \eqref{eq:happy}, with $W=W_{i-1}$, down the filtration \eqref{filh} now gives
\begin{equation}
	W_{i-1} = \sum_{j=1}^{i-1} \uea{\virn} w_j \quad \text{and} \quad C_1(W_{i-1}) = \sum_{j=1}^{i-1} \sum_{n=2}^{\infty} L_{-n} \uea{\virn} w_j.
\end{equation}
It follows that $W_{i-1} / C_1(W_{i-1})$ is spanned by the (images of the) $L_{-1}^m w_j$, with $j<i$ and $m \in \ZZ_{\ge0}$.  By \eqref{eq:happy3}, we have $\dim W_{i-1} / C_1(W_{i-1}) < \infty$ and the proof is complete.

\section{Tensor categories associated to the Virasoro algebra}
Recall that $\ocfin$ denotes the category of finite length $M(c,0)$-modules with composition factors $L(c,h)$ for $h\in H_c$ and note that $\ocfin$ is closed under taking direct sums, generalized submodules, quotient generalized modules and contragredient duals. In this section, we will construct a tensor category structure on $\ocfin$ by verifying that all of the conditions needed in the
Huang-Lepowsky-Zhang logarithmic tensor theory in \cite{HLZ0}--\cite{HLZ8} hold for $\ocfin$. For convenience, we first recall the general constructions and main results in \cite{HLZ0}--\cite{HLZ8}.

\subsection{$P(z)$-tensor product}
In the tensor category theory for vertex operator algebras, the tensor product bifunctors
are not built on the classical tensor product bifunctor for vector spaces. Instead, the central concept underlying the constructions is the notion of {\em $P(z)$-tensor product} \cite{HL3,HLZ3,HLZ4}, where $z$ is a nonzero complex number and $P(z)$ is the Riemann sphere $\overline{\CC}$ with one negatively oriented puncture at $\infty$ and two ordered positively oriented punctures at $z$ and $0$, with local coordinates $1/w$, $w-z$ and $w$, respectively.  We refer to \cite{KaRi} for an expository account that motivates the definition of this tensor product, also known as the fusion product.

\begin{defn}
Let $W_1$, $W_2$ and $W_3$ be generalized modules for a vertex operator algebra $V$. A \emph{$P(z)$-intertwining map of type $\binom{W_3}{W_1\,W_2}$} is a linear map
\begin{equation}
I\colon W_1 \otimes W_2 \longrightarrow \overline{W_3},
\end{equation}
satisfying the following conditions:
\begin{enumerate}[label=\textup{(\roman*)}]
	\item The \emph{lower truncation condition}. For any element $w_{(1)} \in W_1$, $w_{(2)} \in W_2$ and $n \in \CC$,
	\begin{equation}
	\pi_{n-m}(I(w_{(1)}\otimes w_{(2)})) =0 \quad\text{for}\ m \in \NN\ \text{sufficiently large},
	\end{equation}
	where $\pi_n$ is the canonical projection of $\overline{W}$ to the weight subspace $W_{(n)}$
	\item The \emph{Jacobi identity}. For $v\in V$, $w_{(1)}\in W_1$ and $w_{(2)}\in W_2$,
	\begin{align}
	&x_{0}^{-1}\delta \bigg( \frac{x_1-z}{x_{0}}\bigg)
	Y_3(v,x_1)I(w_{(1)}\otimes w_{(2)})\nonumber \\
	\quad &= z^{-1}\delta \bigg( \frac{x_1-x_{0}}{z}\bigg)
	I(Y_1(v,x_1)w_{(1)}\otimes w_{(2)})\nonumber \\
	\quad &+ x_0^{-1}\delta \bigg( \frac{z-x_1}{-x_0}\bigg)I(w_{(1)}\otimes Y_2(v,x_1)w_{(2)}).
	\end{align}
\end{enumerate}
\end{defn}

\begin{remark}The vector space of $P(z)$-intertwining maps of type $\binom{W_3}{W_1\,W_2}$
is isomorphic to the space  $\mathcal{V}_{W_{1}W_{2}}^{W_3}$ of logarithmic intertwining operators (\cref{log:def}) of the same type \cite[Prop.~4.8]{HLZ3}.
\end{remark}

\begin{defn}Let $W_1$ and $W_2$ be generalized $V$-modules. A \emph{$P(z)$-product} of $W_1$ and $W_2$ is a generalized $V$-module $(W_3, Y_3)$ together with a $P(z)$-intertwining map $I_3$ of type $\binom{W_3}{W_1\,W_2}$. We denote it by $(W_3, Y_3; I_3)$ or simply by $(W_3, I_3)$. Let $(W_4, Y_4; I_4)$ be another $P(z)$-product of $W_1$ and $W_2$. A \emph{morphism} from $(W_3, Y_3; I_3)$ to $(W_4, Y_4; I_4)$ is a module map $\eta$ from $W_3$ to $W_4$ such that
\begin{equation}
I_4 = \bar{\eta}\circ I_3,
\end{equation}
where $\bar{\eta}$ is the natural map from $\overline{W_3}$ to $\overline{W_4}$ which extends $\eta$.
\end{defn}

We recall the definition of a $P(z)$-tensor product for a category $\mathcal{C}$ of $V$-modules.  The notion of a $P(z)$-tensor product of $W_1$ and $W_2$ in $\mathcal{C}$ is defined in terms of a universal property as follows.
\begin{defn}For $W_1, W_2 \in \mathcal{C}$, a \emph{$P(z)$-tensor product} of $W_1$ and $W_2$ in $\mathcal{C}$ is a $P(z)$-product  $(W_0, Y_0; I_0)$ with $W_0 \in \mathcal{C}$ such that for any $P(z)$-product $(W, Y; I)$ with $W \in \mathcal{C}$, there is a unique morphism from $(W_0, Y_0; I_0)$ to $(W, Y; I)$. Clearly, a $P(z)$-tensor product  of $W_1$ and $W_2$ in $\mathcal{C}$, if it exists, is unique up to isomorphism. We denote the $P(z)$-tensor product $(W_0, Y_0; I_0)$ by
\begin{equation}
(W_1 \boxtimes_{P(z)} W_2, Y_{P(z)}; \boxtimes_{P(z)})
\end{equation}
and call the object
\begin{equation}
(W_1 \boxtimes_{P(z)} W_2, Y_{P(z)})
\end{equation}
the \emph{$P(z)$-tensor product of $W_1$ and $W_2$}.
\end{defn}

Recall that fusion rules are defined as the dimension of the space of intertwining operators $\mathcal{V}_{W_{1}W_{2}}^{W_{3}}$ (\cite{FHL}, cf.~\cref{log:def}) and can be computed independently of the existence of tensor structures for a vertex algebra \cite{FZ1, FZ2, Li, HY}. However, as one would expect, in the case that the tensor structure of Lepowsky--Huang--Zhang does exist, the space of intertwining operators 
is isomorphic to 
the space of vertex operator algebra homomorphisms from the $P(z)$-tensor product of two modules to a third module. That is, the fusion rules coincide with the fusion product coefficients of the tensor product of two modules. We recall the following result in [HLZ4] that establishes this correspondence:

\begin{prop}\label{isoM}\cite[Prop.~4.17]{HLZ3} Let $W_1$ and $W_2$ be generalized $V$-modules. Suppose that the $P(z)$-tensor product $W_1\boxtimes_{P(z)} W_2$ exists. Then, for $p\in \mathbb{Z}$, we have a natural isomorphism
\begin{align}
\Hom_V(W_1\boxtimes_{P(z)} W_2, W_3) \to 
\mathcal{V}_{W_{1}W_{2}}^{W_{3}}, \quad
\eta \mapsto 
\mathcal{Y}_{\eta, p},
\end{align}
where $\mathcal{Y}_{\eta, p}=\mathcal{Y}_{I,p}$ with $I=\bar{\eta}\, \circ \, \boxtimes_{P(z)}$ 
and, for $w_{(1)}\in W_1$ and $w_{(2)}\in W_2$,
\begin{align}
\mathcal{Y}_{I,p}(w_{(1)},x)w_{(2)} = 
\left. y^{L(0)}x^{L(0)}I(y^{-L(0)}x^{-L(0)}w_{(1)}\otimes y^{-L(0)}x^{-L(0)}w_{(2)})\right\rvert_{y=e^{-\log z-2\pi i p}}.
\end{align}
\end{prop}

We now recall the construction of the $P(z)$-tensor product in \cite{HLZ4}. Let $v \in V$ and let
\begin{equation}
Y_t(v, x) = \sum_{n \in \ZZ}(v \otimes t^n)x^{-n-1} \in (V \otimes \CC[t, t^{-1}])[[x, x^{-1}]].
\end{equation}
Denote by $\tau_{P(z)}$ the action of
\begin{equation}
V \otimes \iota_{+}{\CC}[t,t^{- 1}, (z^{-1}-t)^{-1}]
\end{equation}
on the vector space $(W_1 \otimes W_2)^*$, where $\iota_+$ is the operation of expanding a rational function in the formal variable $t$ in the direction of positive powers of $t$, given by
\begin{align}\label{defofpz}
&\bigg(\tau_{P(z)}\bigg(x_0^{-1}\delta\bigg(\frac{x_1^{-1}-z}{x_0}\bigg)Y_t(v, x_1)\bigg)\lambda\bigg)(w_{(1)}\otimes w_{(2)}) \nonumber \\
&\quad = z^{-1}\delta\bigg(\frac{x_1^{-1}-x_0}{z}\bigg)\lambda(Y_1(e^{x_1L(1)}(-x_1^{-2})^{L(0)}v, x_0)w_{(1)}\otimes w_{(2)}) \nonumber \\
&\quad +\; x_0^{-1}\delta\bigg(\frac{z-x_1^{-1}}{-x_0}\bigg)\lambda(w_{(1)}\otimes Y_2^{\circ}(v, x_1)w_{(2)}),
\end{align}
for $v \in V$, $\lambda \in (W_1 \otimes W_2)^*$, $w_{(1)} \in W_1$ and $w_{(2)} \in W_2$. Denote by $Y_{P(z)}'$ the action of $V \otimes \CC[t,t^{-1}]$ on $(W_1\otimes W_2)^*$ defined by
\begin{equation}
Y_{P(z)}'(v,x) = \tau_{P(z)}(Y_t(v,x)).
\end{equation}
Then, we have the operators $L_{P(z)}'(n)$ for $n \in \ZZ$ defined by
\begin{equation}
Y_{P(z)}'(\omega,x) = \sum_{n \in \ZZ}L_{P(z)}'(n)x^{-n-2}.
\end{equation}

Given two $V$-modules $W_1$ and $W_2$, let $W_1\hboxtr_{P(z)}W_2$ be the vector space consisting of all the elements $\lambda \in (W_1 \otimes W_2)^*$ satisfying the following two conditions.
\begin{enumerate}
\item
\textit{$P(z)$-compatibility condition}:
\begin{enumerate}[label=\textup{(\alph*)}]
\item
\textit{Lower truncation condition}: For all $v \in V$, the formal Laurent series $Y_{P(z)}'(v,x)\lambda$ involves only finitely many negative powers of $x$.
\item
The following formula holds:
\begin{align}
&\tau_{P(z)}\bigg(z^{-1}\delta\bigg(\frac{x_1-x_0}{z}\bigg)Y_t(v,x_0)\bigg)\lambda\nonumber \\
&= z^{-1}\delta\bigg(\frac{x_1-x_0}{z}\bigg)Y_{P(z)}'(v,x_0)\lambda \quad\text{for all}\ v \in V.
\end{align}
\end{enumerate}
\item
\textit{$P(z)$-local grading restriction condition}:
\begin{enumerate}[label=\textup{(\alph*)}]
\item
\textit{Grading condition}: $\lambda$ is a (finite) sum of generalized eigenvectors of $(W_1 \otimes W_2)^*$ for the operator $L_{P(z)}'(0)$.
\item
The smallest subspace $W_{\lambda}$ of $(W_1 \otimes W_2)^*$ containing $\lambda$ and stable under the component operators $\tau_{P(z)}(v\otimes t^n)$ of the operators $Y_{P(z)}'(v,x)$, for $v \in V$ and $n \in \ZZ$, satisfies $\dim (W_{\lambda})_{[n]} < \infty$ and $(W_{\lambda})_{[n+k]} = 0$ for $k \in \ZZ$ sufficiently negative and any $n \in \CC$. Here, the subscripts denote the $\CC$-grading given by the $L_{P(z)}'(0)$-eigenvalues.
\end{enumerate}
\end{enumerate}

\begin{theorem}[\cite{HLZ4}]
The vector space $W_1\hboxtr_{P(z)}W_2$ is closed under the action $Y_{P(z)}'$ of $V$ and the Jacobi identity holds on $W_1\hboxtr_{P(z)}W_2$. Furthermore, the $P(z)$-tensor product of $W_1, W_2 \in \mathcal{C}$ exists if and only if $W_1\hboxtr_{P(z)}W_2$, equipped with $Y_{P(z)}'$, is an object of $\mathcal{C}$. In this case, the $P(z)$-tensor product is the contragredient of $(W_1\hboxtr_{P(z)}W_2, Y_{P(z)}')$.
\end{theorem}

To construct a tensor category structure on $\ocfin$, we first need 
to show that $\ocfin$ is closed under $P(z)$-tensor products. This is an immediate corollary of the following result of Miyamoto.
\begin{theorem}[\cite{Mi}]\label{miyamoto}
Let $W_1, W_2 \in \cC_1$. If $W_3$ is a lower-bounded generalized module such that there exists a surjective intertwining operator of type $\binom{W_3}{W_1\,W_2}$, then $W_3$ is also an object of $\cC_1$. In particular, the $P(z)$-tensor product $W_1 \boxtimes_{P(z)} W_2$ is an object in $\cC_1$.
\end{theorem}

\begin{cor}\label{cor:closedness}
The category $\ocfin$ is closed under taking $P(z)$-tensor products. Namely,
if $W_1, W_2 \in \ocfin$, then $W_1 \btimes_{P(z)} W_2 \in \ocfin$.
\end{cor}
\begin{proof}
By \cref{virc}, $W_1$ and $W_2 \in \ocfin$ are lower-bounded and $C_1$-cofinite, so using \cref{miyamoto} we have that the $P(z)$-tensor product of $W_1$ and $W_2$ is also lower-bounded and $C_1$-cofinite. By \cref{virc} again, we have that $W_1 \btimes_{P(z)} W_2 \in \ocfin$.
\end{proof}

\subsection{Associativity isomorphism}
The associativity isomorphism is the most important ingredient of the tensor category theory of Huang--Lepowsky--Zhang. To prove it, one needs the following convergence and extension property introduced in \cite{HLZ7}.

\begin{defn}
Let $A$ be an abelian group and $\tilde{A}$ an abelian group containing $A$ as a subgroup. Let $V$ be a strongly $A$-graded conformal vertex algebra. We say that the product of the intertwining operators $\mathcal{Y}_1$ and $\mathcal{Y}_2$ satisfies the \emph{convergence and extension property for products} if for any doubly homogeneous elements $w_{(1)} \in W_1^{(\beta_1)}$ and $w_{(2)} \in W_2^{(\beta_2)}$, with $\beta_1, \beta_2 \in \tilde{A}$, and any $w_{(3)} \in W_3$ and $w_{(4)}' \in W_4'$, there exist $r_1, \dots, r_M, s_1, \dots, s_M \in \RR$, $i_1, \dots, i_M, j_1, \dots, j_M \in \NN$ and analytic functions $f_1(z), \dots, f_M(z)$ on $\abs{z} < 1$ (for some $M \in \NN$) satisfying
\begin{equation}
\wt w_{(1)} + \wt w_{(2)} + s_k > N,\quad \text{for each}\ k = 1, \dots, M,
\end{equation}
where $N \in \ZZ$ depends only on $\mathcal{Y}_1$, $\mathcal{Y}_2$ and $\beta_1 + \beta_2$, such that
\begin{equation}
\Bigl.\bigl\langle w_{(4)}', \mathcal{Y}_1(w_{(1)},  x_1)\mathcal{Y}_2(w_{(2)}, x_2)w_{(3)} \bigr\rangle_{W_4} \Bigr\rvert_{x_1 = z_1,\ x_2 = z_2}
\end{equation}
is absolutely convergent on $\abs{z_1} > \abs{z_2} > 0$ and may be analytically extended to the multivalued analytic function
\begin{equation}
\sum_{k = 1}^M z_2^{r_k}(z_1 - z_2)^{s_k}(\log z_2)^{i_k}(\log (z_1 - z_2))^{j_k}f_k\left(\frac{z_1 - z_2}{z_2}\right)
\end{equation}
in the region $\abs{z_2} > \abs{z_1 - z_2} > 0$.
\end{defn}

\begin{defn}
We say that the iterate of the intertwining operators $\mathcal{Y}^1$ and $\mathcal{Y}^2$ satisfies the \emph{convergence and extension property for iterates} if for any doubly homogeneous elements $w_{(2)} \in W_2^{(\beta_2)}$, $w_{(3)} \in W_3^{(\beta_3)}$ with $\beta_2, \beta_3 \in \tilde{A}$, and any $w_{(1)} \in W_1$ and $w_{(4)}' \in W_4'$, there exist $\tilde{r}_1, \dots, \tilde{r}_M, \tilde{s}_1, \dots, \tilde{s}_{\widetilde{M}} \in \RR$, $\tilde{i}_1, \dots, \tilde{i}_{\widetilde{M}}, \tilde{j}_1, \dots, \tilde{j}_{\widetilde{M}} \in \NN$ and analytic functions $\tilde{f}_1(z), \dots, \tilde{f}_{\widetilde{M}}(z)$ on $\abs{z} < 1$ (for some $\widetilde{M} \in \NN$) satisfying
\begin{equation}
\wt w_{(2)} + \wt w_{(3)} + \tilde{s}_k > \widetilde{N},\quad \text{for each}\ k = 1, \dots, \widetilde{M},
\end{equation}
where $\widetilde{N} \in \ZZ$ depends only on $\mathcal{Y}^1$, $\mathcal{Y}^2$ and $\beta_2 + \beta_3$, such that
\begin{equation}
\Bigl.\bigl\langle w_{(4)}', \mathcal{Y}^1(\mathcal{Y}^2(w_{(1)}, x_0)w_{(2)}, x_2)w_{(3)} \bigr\rangle_{W_4} \Bigr\rvert_{x_0 = z_1-z_2,\ x_2 = z_2}
\end{equation}
is absolutely convergent on $\abs{z_2} > \abs{z_1-z_2} > 0$ and may be analytically extended to the multivalued analytic function
\begin{equation}
\sum_{k = 1}^{\widetilde{M}} z_1^{\tilde{r}_k}z_2^{\tilde{s}_k}(\log z_1)^{\tilde{i}_k}(\log z_2)^{\tilde{j}_k}\tilde{f}_k\left(\frac{z_2}{z_1}\right)
\end{equation}
in the region $\abs{z_1} > \abs{z_2} > 0$.
\end{defn}

The convergence and extension property is part of the known sufficient conditions \cite[Thm.~10.3]{HLZ6}, \cite[Thm.~11.4]{HLZ7} and \cite[Thm.~3.1]{H6} for the existence of the associativity isomorphism.
\begin{theorem}[\cite{HLZ6,HLZ7,H6}] \label{assiso}
Let $V$ be a vertex operator algebra satisfying the following conditions:
\begin{enumerate}
\item For any two modules $W_1$ and $W_2$ in $\mathcal{C}$ and any $z \in \CC^{\times}$, if the generalized $V$-module $W_{\lambda}$ is generated by a generalized $L_{P(z)}'(0)$-eigenvector $\lambda \in (W_1 \otimes W_2)^{*}$ satisfying the $P(z)$-compatibility condition is lower-bounded,
then $W_{\lambda}$ is an object of $\mathcal{C}$. \label{Pzcomp}
\item The convergence and extension property holds for either the product or the iterates of intertwining operators for $V$. \label{convext}
\end{enumerate}
Then, for any $V$-modules $W_1, W_2$ and $W_3$ and any $z_1, z_2 \in \CC$ satisfying $\abs{z_1}>\abs{z_2}>\abs{z_1-z_2}>0$, there is a unique isomorphism
\begin{equation}
	\mathcal{A}_{P(z_1), P(z_2)}^{P(z_1-z_2), P(z_2)} \colon W_1 \boxtimes_{P(z_1)} (W_2 \boxtimes_{P(z_2)} W_3) \to (W_1 \boxtimes_{P(z_1-z_2)} W_2) \boxtimes_{P(z_2)} W_3
\end{equation}
such that for $w_{(1)} \in W_1$, $w_{(2)} \in W_2$ and $w_{(3)} \in W_3$, one has
\begin{equation}
\overline{\mathcal{A}}_{P(z_1), P(z_2)}^{P(z_1-z_2), P(z_2)}(w_1 \boxtimes_{P(z_1)} (w_2 \boxtimes_{P(z_2)} w_3))  = (w_1 \boxtimes_{P(z_1-z_2)} w_2) \boxtimes_{P(z_2)} w_3,
\end{equation}
where
\begin{equation}
\overline{\mathcal{A}}_{P(z_1), P(z_2)}^{P(z_1-z_2), P(z_2)}\colon \overline{W_1 \boxtimes_{P(z_1)} (W_2 \boxtimes_{P(z_2)} W_3)} \rightarrow \overline{(W_1 \boxtimes_{P(z_1-z_2)} W_2) \boxtimes_{P(z_2)} W_3}
\end{equation}
is the canonical extension of $\mathcal{A}_{P(z_1), P(z_2)}^{P(z_1-z_2), P(z_2)}$.
\end{theorem}

Condition \ref{convext} in \cref{assiso} is guaranteed by the $C_1$-cofiniteness condition.
\begin{prop}\label{vice}
The convergence and extension property for products and iterates holds for $\ocfin$.
\end{prop}
\begin{proof}
It follows from \cite[Thm.~11.8]{HLZ7} that if all the objects $W \in \ocfin$ are $C_1$-cofinite and satisfy $\dim \coprod_{\Re(n) < r} W_{[n]} < \infty$, for any $r \in \RR$, then the convergence
and extension properties for products and iterates of intertwining operators hold.  The $C_1$-cofiniteness is guaranteed by \cref{virc} while the second condition is obvious because objects in $\ocfin$ have finite lengths.
\end{proof}

\begin{theorem}\label{thm:associativity}
The associativity isomorphism holds for the category $\ocfin$.
\end{theorem}
\begin{proof}
Given \cref{vice}, we only need to show that Condition \ref{Pzcomp} in \cref{assiso} holds in $\ocfin$. Let $W_{\lambda}$ be the lower-bounded generalized $V$-module constructed in \cref{assiso}.

The inclusion map $W_{\lambda} \subset (W_1\otimes W_2)^*$ intertwines the action of $V \otimes \iota_{+}{\CC}[t,t^{- 1}, (z^{-1}-t)^{-1}]$ given by $\tau_{P(z)}$. Therefore, by \cite[Prop.~5.24]{HLZ4}, it corresponds to a $P(z)$-intertwining map $I$ of type $\binom{W_{\lambda}'}{W_1\,W_2}$ such that
\begin{equation}
\lambda(w_1\otimes w_2)=\langle \lambda, I(w_1\otimes w_2)\rangle.
\end{equation}
By \cite[Prop.~4.8]{HLZ3}, there is an intertwining
operator $\mathcal{Y}$ of type $\binom{W_{\lambda}'}{W_1\,W_2}$ such that
\begin{equation}
\mathcal{Y}(w_1, z)w_2 = I(w_1\otimes w_2).
\end{equation}
Note that this intertwining operator is surjective since
\begin{equation}
\langle \lambda, \mathcal{Y}(w_1, z)w_2\rangle = \langle \lambda, I(w_1\otimes w_2)\rangle = \lambda(w_1\otimes w_2).
\end{equation}
By \cref{miyamoto}, $W_{\lambda}'$ is $C_1$-cofinite since $W_1$ and $W_2$ are $C_1$-cofinite. Then, since $W_{\lambda}' $ is $C_1$-cofinite and lower-bounded, it follows from \cref{virc} that $W_{\lambda}' \in \ocfin$ and thus $W_{\lambda} \in \ocfin$.
\end{proof}

By Corollary \ref{cor:closedness}, Theorem \ref{thm:associativity} and \cite[Thm.~12.15]{HLZ8}, we obtain the main \lcnamecref{thm:tensor} of this section.
\begin{theorem}\label{thm:tensor}
The category $\ocfin$ has a braided tensor category structure.
\end{theorem}

\begin{remark}
There are many open conjectures about the representation categories of the singlet and triplet algebras \cite{CGan, CM, CMR, CGR, RiW}, but the most basic one is the existence of a rigid vertex tensor category structure on the category of finite-length modules \cite{CMR}. Since the singlet (and triplet) vertex operator algebras are objects in the Ind-completion of $\ocfin$ for $c = 1 - 6(p-1)^2/p$, so $t = 1/p$, our results may be used to study the vertex tensor category of modules containing all known indecomposable but reducible modules for the singlet algebra.
\end{remark}

\section{Rigidity for generic central charge} \label{sec:rigid}
In this section, we will study the category $\ocfin$ for generic central charges $c = 13 - 6t - 6t^{-1}$, meaning that $t \notin \QQ$, and prove that it is rigid. The simple objects of this category have the form $L(c, h_{r,s})$ for $r, s \in \ZZ_{\ge1}$.
Set $t = k+2$, so that $k \notin \QQ$, and recall the notation $L_{r,s} = L(c, h_{r,s})$ for $r, s \in \ZZ_{\ge1}$ from \cref{sec:vir}. The condition that $c$ is generic is understood to be in force for the rest of the \lcnamecref{sec:rigid} unless otherwise noted.

\subsection{Semisimplicity}
We start by establishing that $\ocfin$ is semisimple for generic central charges.

\begin{lemma}\label{lemma:noext}
If $h \neq h' \in H_c$, then
\begin{equation}
\Ext^1_{\ocfin}(L(c, h'), L(c, h)) = 0.
\end{equation}
\end{lemma}
\begin{proof}
Suppose there is a short exact sequence
\begin{equation}\label{exseq}
0 \rightarrow L(c,h) \rightarrow M \rightarrow L(c,h') \rightarrow 0.
\end{equation}
If $\Im(h) \neq \Im(h')$, then this sequence obviously splits, so we may assume that $\Im(h) = \Im(h')$.

If $\Re(h') < \Re(h)$, then there is a \hw\ vector of conformal weight $h'$ in $M$. From \cref{exhaustiveembedding}, the singular vector $v_{h'}$ either generates a Verma module $V(c, h')$ or its simple quotient $L(c, h')$. But, it has to be $L(c, h')$ because $V(c, h')$ is not in $\ocfin$. The exact sequence \eqref{exseq} therefore splits.

If $\Re(h') > \Re(h)$, we consider the contragredient dual of $M$, recalling that $L(c, h)$ and $L(c,h')$ are self-dual, arriving at the short exact sequence
\begin{equation}
0 \rightarrow L(c,h') \rightarrow M' \rightarrow L(c,h) \rightarrow 0.
\end{equation}
The previous argument then gives $M' \cong L(c, h) \oplus L(c, h')$, hence $M \cong L(c, h) \oplus L(c, h')$ too.
\end{proof}

\begin{theorem}\label{thm:semisimple}
The category $\ocfin$ is semisimple for generic central charges.
\end{theorem}
\begin{proof}
 \cref{lemma:noext} shows that there are no extensions between non-isomorphic simple objects.  But, \cite[Thm.~6.4]{R} and \cite[Prop.~7.5]{KRi} (see also \cite[Lem.~5.2.2]{GK} for a high-powered approach) have shown that there are no self-extensions of $L(c,h)$ for generic central charges.
\end{proof}
\noindent Despite claims to the contrary in the literature, it is interesting that certain simples do admit self-extensions for non-generic central charges.  The $L(c,h)$ that do are classified in \cite{KRi}.

\subsection{Fusion rules} \label{sec:fusion}
We next determine the fusion rules of the category $\ocfin$, for generic $c$, using the Zhu algebra tools developed in \cite{FZ1,FZ2}, see also \cite{Li}.  From \cite{FZ1,W}, we know that the Zhu algebra of $M(c,0)$ is $A(M(c,0)) \cong \CC[x]$, where $x=[\omega]$.
Moreover, we have $A(V(c,h)) \cong \CC[x,y]$ as a $\mathbb{C}[x]$-bimodule \cite{DMZ,Li}.

The fusion rules involving the simple $M(c,0)$-modules $L_{1,s}$ were computed by I.~Frenkel and M.~Zhu in \cite[Prop.~2.24]{FZ2} using the explicit singular vector formula of Benoit and Saint-Aubin \cite{BSA}.

\begin{remark}
Note that it follows from \cref{isoM}, and the discussion that precedes it, that the fusion rules computed by  I.~Frenkel and M.~Zhu in \cite{FZ2} are indeed, in light of \cref{thm:tensor,thm:semisimple}, the fusion product coefficients in the vertex tensor category.
\end{remark}

Combined with our \cref{thm:tensor,thm:semisimple}, their result may be phrased as follows.
\begin{theorem}[\cite{FZ2}] \label{thm:1sfusion}
	Let $c$ be generic.  Then, for $s_1, s_2, s_3 \in \ZZ_{\ge1}$, we have
	\begin{equation} \label{fr:L1s}
		L_{1,s_1} \btimes L_{1,s_2} \cong \bigoplus_{\mathclap{\substack{s_3 = \abs{s_1-s_2}+1 \\ s_3=s_1+s_2-1 \bmod{2}}}}^{s_1+s_2-1} L_{1,s_3}.
	\end{equation}
\end{theorem}
\noindent The obvious analogue for $L_{r_1,1} \btimes L_{r_2,1}$ also holds.

Our interest, however, is in the fusion of $L_{r,1}$ with $L_{1,s}$. This can also be attacked using the same methods, in particular the results \cite[Lems.~2.9 and 2.14]{FZ2} which we quote for convenience.
\begin{prop}[\cite{FZ2}] \label{prop:FZ2}
	\leavevmode
	\begin{enumerate}
		\item As an $A(M(c,0))$-bimodule, we have
		\begin{equation}
			A(L_{r,s}) \cong \frac{\CC[x,y]}{\left\langle f_{r,s}(x,y) \right\rangle},
		\end{equation}
		where $f_{r,s}(x,y)$ is the image of the non-cyclic singular vector of $V_{r,s}$ in $A(V_{r,s})$.
		\item Let $\overline{n}$ denote the residue of $n \in \ZZ$ modulo $2$.  Then, for $r,s \in \ZZ_{\ge1}$, we have
		\begin{equation}
			f_{r,1} = g_r' g_{r-2}' \dots g_{\overline{r}}'
			\quad \text{and} \quad
			f_{1,s} = g_s g_{s-2} \dots g_{\overline{s}},
		\end{equation}
		where $g_1(x,y) = g_1'(x,y) = x-y$ and, for $r,s \in \ZZ_{\ge2}$,
		\begin{equation}
			\begin{aligned}
				g_r'(x,y) &= (x-y-h_{r,1})(x-y-h_{-(r-2),1}) - (r-1)^2 t y \\
				\text{and} \quad
				g_s(x, y) &= (x-y-h_{1,s})(x-y-h_{1,-(s-2)}) - (s-1)^2 t^{-1} y.
			\end{aligned}
		\end{equation}
	\end{enumerate}
\end{prop}

Specialising \cite[Lem.~2.22]{FZ2} to the case we are interested in, we arrive at the following characterisation of the fusion rules.
\begin{prop}[\cite{FZ2}]\label{FZ}
$\cN_{L_{r,1} L_{1,s}}^{L_{r',s'}} = 0$ unless
\begin{subequations} \label{e}
\begin{align}
\label{e2}f_{r,1}(h_{r',s'}, h_{1,s}) &= 0,\\
\label{e1}f_{1,s}(h_{r',s'}, h_{r,1}) &= 0\\
\label{e3}\text{and} \quad f_{r',s'}(h_{1,s}, h_{r,1}) &= 0,
\end{align}
\end{subequations}
in which case, $\cN_{L_{r,1} L_{1,s}}^{L_{r',s'}} = 1$.
\end{prop}
\noindent Fix now $r,s \in \ZZ_{\ge1}$ and consider first the $h_{r',s'}$ that solve \eqref{e2}.  Direct calculation shows that the solutions of $g_r'(x,h_{1,s}) = 0$ are $x=h_{r,s}$ and, if $r>1$, $x=h_{-r+2,s}$.  An easy calculation verifies that
\begin{equation} \label{eq:confwtseq}
	h_{r,s} = h_{r',s'} \iff (r,s) = (r',s')\ \text{or}\ (-r',-s'),
\end{equation}
since $t \notin \QQ$.  It follows that $h_{-r+2,s} \notin H_c$, for $r>1$, and so these solutions of \eqref{e2} correspond to simple Verma modules $V_{-r+2,s} = L_{-r+2,s} \notin \ocfin$ and hence cannot appear in the fusion of $L(r,1)$ and $L(1,s)$, by \cref{thm:tensor}.  The viable solutions are therefore $h_{r',s'} = h_{r,s}, h_{r-2,s}, \dots, h_{\overline{r},s}$.

However, the same analysis gives the viable solutions of \eqref{e1} as $h_{r',s'} = h_{r,s}, h_{r,s-2}, \dots, h_{r,\overline{s}}$.  Appealing to \eqref{eq:confwtseq}, we conclude that there is a unique solution to \cref{e2,e1}:  $h_{r',s'} = h_{r,s}$.
We can now state our main fusion rule.
\begin{theorem}\label{fusionrule}
Let $c$ be generic. Then, for $r,s \in \ZZ_{\ge1}$, we have
\begin{equation} \label{fr:Lrs}
L_{r,1}\btimes L_{1,s} = L_{r, s}.
\end{equation}
\end{theorem}
\begin{proof}
Since $\ocfin$ is tensor (\cref{thm:tensor}) and semisimple (\cref{thm:semisimple}), $L_{r,1}\btimes L_{1,s}$ decomposes as a finite direct sum of simples.  The previous arguments establish that the only possibilities for this decomposition are $0$ or $L_{r,s}$, with the deciding condition whether \eqref{e3} is satisfied (for $r'=r$ and $s'=s$).  Unfortunately, computing the left-hand side of \eqref{e3} directly is difficult.

Instead, note that combining $L_{r,1}\btimes L_{1,s} = 0$ with \cref{thm:1sfusion} (and associativity) gives
\begin{align}
	0 &\cong (L_{r,1} \btimes L_{1,s}) \btimes L_{1,s} \cong L_{r,1} \btimes (L_{1,s} \btimes L_{1,s}) \cong L_{r,1} \btimes (L_{1,1} \oplus \dots \oplus L_{1,2s-1}) \notag \\
	&\cong L_{r,1} \oplus \cdots,
\end{align}
a contradiction.  It follows that $L_{r,1}\btimes L_{1,s} \neq 0$ and the proof is complete.
\end{proof}
\noindent We remark that this \lcnamecref{fusionrule} proves, for generic $c$, a well-known conjecture of physicists.  Of course, physicists are more interested in the non-generic version of this conjecture \cite[Eq.~(4.34)]{MRR}.

\subsection{Coset realizations of the Virasoro algebra}
The next ingredient for our proof of the rigidity of $\ocfin$ is the well-known coset realization of the Virasoro vertex operator algebra \cite{GKO, ACL}.  We review this here for general simply laced Lie algebras before specializing to $\sl_2$.

Let $P_+$ be the set of dominant integral weights of a simple complex Lie algebra $\frg$ and let $Q$ be its root lattice. Let $\widehat{\frg} = \frg[t,t^{-1}] \oplus \CC K$ be the affinization of $\frg$.  For $\ell \in \CC$ and $\l \in P_+$, let $V_{\ell}(\l) = U(\widehat{\frg})\otimes_{U(\frg[t]\oplus \CC{K})}E^{\l}$, where $E^{\l}$ is the finite-dimensional $\frg$-module of highest weight $ \lambda$ regarded as a $\frg[t]\oplus \CC{K}$-module on which $t\frg[t]$ acts trivially and $K$ acts as multiplication by the level $\ell$. Denote by $L_{\ell}(\l)$ the simple quotient of $V_{\ell}(\l)$. The modules $V_{\ell}(0)$ and $L_{\ell}(0)$ admit, for $\ell \neq -h^{\vee}$, a vertex operator algebra structure \cite{FZ1}. To emphasize the dependence of these vertex operator algebras on $\frg$, we denote them by $V_{\ell}(\frg)$ and $L_{\ell}(\frg)$.

It is known that the category $KL_k$ of ordinary $L_k(\frg)$-modules has a rigid braided tensor category structure for all $k$ such that $k + h^{\vee} \notin \QQ_+$, $k \in \NN$, or $k$ is admissible \cite{KL1}--\cite{KL5}, \cite{HL5,CHY}.  When $k \in \NN$, the simple objects of $KL_k$ are rather the $L_k(\l)$ with $\l \in P^k_+ = \left\{\l\in P_+ \mid (\l, \th)\le k\right\}$, where $\th$ is the longest root of $\frg$.  Here, we study the generic case $k \notin \QQ$, for which $V_k(\frg) = L_k(\frg)$ and the simple objects of $KL_k$ are the $L_k(\l)$ with $\l \in P_+$.

Let $W^k(\frg)$ be the $W$-algebra associated with $\frg$ and a 
principal nilpotent element of $\frg$ at level $k$. Let $W_k(\frg)$ be the unique simple quotient of $W^k(\frg)$. We denote by $\chi_{\l}$ the central character associated to the weight $\l \in P_+$ and let $\bM_k(\chi_{\l})$ be the Verma module of $W^k(\frg)$ with highest weight $\chi_{\l}$ and $\bL_k(\chi_{\l})$ its unique simple (graded) quotient.

By specializing to $\frg = \sl_2$ and identifying $P_+$ for $\sl_2$ with $\ZZ_{\ge 0}$, we have $W^k(\sl_2) = M(c,0)$ and $W_{k}(\sl_2) = L(c, 0)$, for $c = 13 - 6t- 6t^{-1}$ with $t=k+2$. For $\mu, \nu \in P_+$, we moreover have
\begin{equation}
\bL_{k}(\chi_{\mu - 2(k +h^{\vee})\nu}) = L_{\mu+1,\nu+1} \in \ocfin.
\end{equation}

\begin{theorem}[\cite{ACL}]\label{cosetvoa}
Suppose that $\frg$ is a simply-laced Lie algebra and $\ell \notin \QQ$. Define $k \notin \QQ$ by
\begin{equation}\label{defl}
k + h^{\vee} = \frac{\ell+h^{\vee}}{\ell+h^{\vee}+1}.
\end{equation}
Then, for $\l \in P^1_+$ and $\mu \in P_+$, we have
\begin{subequations} \label{gencosetdecomps}
\begin{equation}
L_1(\l) \otimes V_{\ell}(\mu) \cong \bigoplus_{\substack{\nu \in P_+, \\ \l+\mu-\nu \in Q}} V_{\ell+1}(\nu) \otimes \bL_{k}(\chi_{\mu - 2(k +h^{\vee})\nu})
\end{equation}
as $V_{\ell+1}(\frg)\otimes W^{k}(\frg)$-modules. In particular,
\begin{equation}
L_1(\frg) \otimes V_{\ell}(\frg) \cong \bigoplus_{\nu \in P_+\cap Q} V_{\ell+1}(\nu) \otimes \bL_{k}(\chi_{- 2(k +h^{\vee})\nu}).
\end{equation}
\end{subequations}
\end{theorem}
\noindent By specializing \cref{cosetvoa} to $\sl_2$, the decompositions \eqref{gencosetdecomps} become 
\begin{subequations} \label{coset}
\begin{equation}\label{coset1}
L_1(\l) \otimes V_{\ell}(\mu) \cong \bigoplus_{\substack{\nu=0 \\ \l+\mu = \nu \bmod{2}}}^{\infty} V_{\ell+1}(\nu) \otimes L_{\mu+1,\nu+1},
\end{equation}
for $\l \in P^1_+ = \{0,1\}$ and $\mu \in P_+ = \NN$, and
\begin{equation}\label{coset2}
L_1(\sl_2) \otimes V_{\ell}(\sl_2) \cong \bigoplus_{\nu=0}^{\infty} V_{\ell+1}(2\nu) \otimes L_{1,2\nu+1}.
\end{equation}
\end{subequations}

\subsection{Categories of vertex operator algebra extensions}

The last ingredient we need is the theory of tensor categories of vertex operator algebra extensions \cite{CKM}.  Let $\mathcal{C}$ be a braided tensor category with tensor unit $\one$, braiding $c$ and associativity isomorphism $a$.

\begin{defn}
An \emph{associative algebra} in $\mathcal{C}$ is a triple $(A, \mu_A, \iota_A)$ with $A \in \mathcal{C}$ and $\mu_A\colon A \btimes A \rightarrow A$, $\iota_A\colon \one \rightarrow A$ morphisms in $\mathcal{C}$ satisfying the following axioms.
\begin{enumerate}
\item
Unit: $A = \one \btimes A \xrightarrow{\iota_A\btimes \id_A} A \btimes A \xrightarrow{\mu_A} A$ and $A = A \btimes \one \xrightarrow{\id_A\btimes \iota_A} A \btimes A \xrightarrow{\mu_A} A$ are both $\id_A$.
\item
Associativity: $\mu_A\circ (\id_A\btimes \mu_A) = \mu_A \circ (\mu_A \btimes \id_A)\circ a_{A,A,A}\colon A \btimes (A \btimes A) \rightarrow A$.
\end{enumerate}
An associative algebra $A \in \mathcal{C}$ is said to be \emph{commutative} if $\mu_A\circ c_{A, A} = \mu_A\colon A\btimes A \rightarrow A$.
\end{defn}

\begin{defn}
For an associative algebra $A$ in $\mathcal{C}$, define $\mathcal{C}_A$ to be the category of pairs $(X, \mu_X)$ for which $X \in \mathcal{C}$ and $\mu_X \in \Hom_{\mathcal{C}}(A\otimes X, X)$ satisfy the following.
\begin{enumerate}
\item
Unit: $X = \one \btimes X \xr{\iota_A \btimes \id_X} A\btimes X \xr{\mu_X} X$ is $\id_X$.
\item
Associativity: $\mu_X\circ (\id_A \btimes \mu_X) = \mu_X \circ (\mu_A \btimes \id_X)\circ a_{A,A,X}\colon A \otimes (A \btimes X) \rightarrow X$.
\end{enumerate}
Define $\mathcal{C}_A^0$ to be the full subcategory of $\cC_A$ consisting of \emph{local modules}: those objects $(X, \mu_X)$ such that $\mu_X\circ c_{X, A}\circ c_{A, X} = \mu_X$.
\end{defn}

When $A$ is commutative, the category $\cC_A$ is naturally a tensor category with tensor product $\btimes_A$ and unit object $A$, while the subcategory $\cC_A^0$ is braided tensor (see for example \cite{KO} --- the case in which $A$ is an object of the direct sum completion of $\cC$, that is a countably infinite direct sum of objects, is addressed in \cite{AR, CGR, CKM2}).

Let $V$ be a vertex operator algebra and $\cC$ be a category of $V$-modules with a natural tensor category structure. The following \lcnamecref{voaext} allows one to study vertex operator algebra extensions using abstract tensor category theory.

\begin{theorem}[\cite{HKL, CKM}]\label{voaext}
A vertex operator algebra extension $V \subset A$ in $\cC$ is equivalent to a commutative associative algebra in the braided tensor category $\cC$ with trivial twist and injective unit. Moreover, the category of modules in $\cC$ for the extended vertex operator algebra $A$ is braided tensor equivalent to the category of local $\cC$-algebra modules $\cC_A^0$ via the induction functor $\cF_A$.
\end{theorem}
\noindent We recall that the \emph{twist} $\theta_X$ of an object $X$ in a category of modules over a vertex operator algebra is given by the action of $e^{2\pi i L_0}$.  To say that $A$ has trivial twist above therefore means that $L_0$ acts semisimply on $A$ with integer conformal weights.

\subsection{Rigidity} \label{sec:rigidproof}
It is time to put all these ingredients together to prove the rigidity of $\ocfin$ for all generic central charges.  We first use \cref{voaext} to prove a braided tensor equivalence between the full subcategory $\ocleft \subset \ocfin$, generated by the simple Virasoro modules $L_{\mu,1}$ with $\mu \in \ZZ_{\ge1}$, and a simple current twist of the category $KL_{\ell}$ of ordinary $L_{\ell}(\sl_2)$-modules.  Here, $\ell$ is related to $k$ by \eqref{defl}, hence to $t=k+2$ and the (generic) central charge $c=13-6t-6t^{-1}$.

Let
\begin{equation}
A = \bigoplus_{\nu=0}^{\infty} V_{\ell+1}(2\nu) \otimes L_{1,2\nu+1} \cong L_1(\sl_2) \otimes V_{\ell}(\sl_2).
\end{equation}
Then, $A$ is a commutative associative algebra object in the direct sum completion of $KL_{\ell+1}\btimes \ocfin$.  Denote by $\cF_A \colon KL_{\ell+1}\btimes \ocfin \rightarrow \left(KL_{\ell+1}\btimes \ocfin\right)_A$ the induction functor.
\begin{lemma} \label{prop:indfaith}
	The restriction of the induction functor $\cF_A$ to the full subcategory $\ocleft$ is fully faithful.
\end{lemma}
\begin{proof}
	Using the Virasoro fusion rule \eqref{fr:Lrs} and the specializations \eqref{coset} of \cref{cosetvoa}, we obtain
	\begin{align}\label{eq:liftiso}
		\cF_A(V_{\ell+1}(\sl_2) \otimes L_{\mu+1,1})
		&= \biggl[ \bigoplus_{\nu=0}^{\infty} V_{\ell+1}(2\nu) \otimes L_{1,2\nu+1} \biggr] \btimes \biggl[ V_{\ell+1}(\sl_2) \otimes L_{\mu+1,1} \biggr] \notag \\
		&\cong \bigoplus_{\nu=0}^{\infty} V_{\ell+1}(2\nu) \otimes L_{\mu+1,2\nu+1}
		\cong L_1(\ov{\mu}) \otimes V_{\ell}(\mu),
	\end{align}
	where $\ov{\mu}$ is the residue of $\mu \in \ZZ_{\ge 0}$ modulo $2$.  \emph{A priori}, the last isomorphism is a $KL_{\ell+1}\btimes \ocfin$-morphism.  However, Frobenius reciprocity gives, for $\nu \in \ZZ_{\ge0}$ and $\rho \in \{0,1\}$,
	\begin{align}
		&\Hom_{\left(KL_{\ell+1}\btimes \ocfin\right)_A} \Bigl( \cF_A(V_{\ell+1}(\sl_2) \otimes L_{\mu+1,1}), L_1(\rho) \otimes V_{\ell}(\nu) \Bigr)\notag\\
		&\quad \cong \Hom_{KL_{\ell+1}\btimes \ocfin} \Big( V_{\ell+1}(\sl_2) \otimes L_{\mu+1,1},\ \bigoplus_{\mathclap{\substack{\l = 0 \\ \l = \nu+\rho \bmod{2}}}}^{\infty}\ V_{\ell+1}(\l) \otimes L_{\nu+1,\l+1} \Big)\notag\\
		&\quad \cong
		\begin{cases*}
			\CC & if $\nu = \mu$ and $\rho=\ov{\mu}$, \\
			0 & otherwise.
		\end{cases*}
	\end{align}
	In other words, the isomorphism \eqref{eq:liftiso} is actually an isomorphism in $KL_1 \btimes KL_{\ell}$.
\end{proof}

Restricting $\cF_A$ to $\ocleft$, its image is the full subcategory of $KL_1 \btimes KL_{\ell}$ whose simple objects have the form $L_1(\ov{\mu}) \otimes V_{\ell}(\mu)$. Denote this category by $\left(KL_1 \btimes KL_{\ell}\right)_0$. This category is a rigid semisimple tensor category with the same simple objects and fusion rules as the category of finite dimensional $\sl_2$-modules (\cite{KL1}-\cite{KL5}, \cite{L}).  This proves the following \lcnamecref{prop:leftequiv}.
\begin{prop} \label{prop:leftequiv}
	$\ocleft$ and $\left(KL_1 \btimes KL_{\ell}\right)_0$ are braided-equivalent tensor categories.  In particular, $\ocleft$ is rigid.
\end{prop}

We now come to the main \lcnamecref{thm:rigid} of this section.
\begin{theorem}\label{thm:rigid}
 The category $\ocfin$ is rigid for generic central charges.
\end{theorem}
\begin{proof}
As $\ocleft$ is rigid, the simple objects $L_{\mu,1}$, with $\mu \in \ZZ_{\ge1}$ 
are rigid.  An identical argument shows that the $L_{1,\nu}$ with $\nu \in \ZZ_{\ge1}$ are likewise rigid.  It follows now from \cref{fusionrule} that $L_{\mu, \nu} \cong L_{\mu,1} \btimes L_{1,\nu}$ is rigid.  Since $\ocfin$ is semisimple (\cref{thm:semisimple}), this completes the proof.
\end{proof}

\begin{remark}
Braided tensor equivalences between categories of modules for $W$-algebras and affine vertex operator algebras may be viewed as reformulations of problems in the quantum geometric Langlands program.  For example, \cref{prop:leftequiv} is the case $\mathfrak{g}=\mathfrak{sl}_2$, $N=1$ and $\beta$ generic of \cite[Conj.~6.4]{AFO}, up to a simple current twist (see also \cite[Rem.~7.2]{C1}).
\end{remark}

We finish by demonstrating that $\ocfin$ is moreover non-degenerate.  Recall that in a rigid braided tensor category $\cC$, an object $T \in \cC$ is called \emph{transparent} if it has trivial monodromy with every other object of $\cC$, that is if $c_{X, T}\circ c_{T, X} = \id_{T \btimes X}$ for all $X \in \cC$.  $\cC$ is said to be \emph{non-degenerate} if the only transparent objects are finite direct sums of the tensor unit. In our case, non-degeneracy follows easily from the balancing axiom
\begin{equation} \label{eq:balance}
c_{X, T}\circ c_{T, X} = \left(\theta_X^{-1} \boxtimes \theta_T^{-1}\right) \circ \theta_{X \boxtimes T},
\end{equation}
where we recall that $\theta_X$ denotes the twist of $X \in \cC$.

\begin{prop}\label{prop:nondegenerate}
The category $\ocfin$ is non-degenerate for generic central charges.
\end{prop}
\begin{proof}
	Set $T= L_{r,s}$ with $s \neq 1$.  We then take $X=L_{1,2}$ so that the fusion rules \eqref{fr:L1s} and \eqref{fr:Lrs} give
	\begin{align}
		T \btimes X
		&\cong (L_{r,1} \btimes L_{1,s}) \btimes L_{1,2}
		\cong L_{r,1} \btimes (L_{1,s} \btimes L_{1,2})
		\cong L_{r,1} \btimes (L_{1,s-1} \oplus L_{1,s+1}) \notag \\
		&\cong L_{r,s-1} \oplus L_{r,s+1}.
	\end{align}
	As the twist coincides with the action of $e^{2 \pi i L_0}$, \eqref{eq:balance} shows that $c_{X, T}\circ c_{T, X}$ acts as multiplication by $e^{2 \pi i (h_{r,s\pm1}-h_{r,s}-h_{1,2})}$ on $L_{r,s\pm1}$.  The ratio $e^{2 \pi i s/t}$ of these eigenvalues is never $1$, hence $T$ is not transparent.  The case $r \neq 1$ is completely analogous.
\end{proof}

\flushleft
\bibliographystyle{alpha}

\end{document}